\documentclass[11pt]{amsart}
\usepackage{hyperref}
\usepackage{amsfonts}
\usepackage{amsmath}
\usepackage{amssymb}
\usepackage{amscd}
\usepackage{graphics}
\usepackage{latexsym}
\usepackage{color}
\usepackage{epsfig}
\usepackage{amsopn}
\usepackage{xypic}
\usepackage{mathrsfs}
\usepackage{array}

\usepackage{booktabs}
\usepackage{longtable}
\theoremstyle{plain}
\newtheorem{lemma}{Lemma}[section]
\newtheorem*{theorem*}{Theorem}
\newtheorem*{lemma*}{Lemma}
\newtheorem*{proposition*}{Proposition}
\newtheorem*{conjecture*}{Conjecture}
\newtheorem*{corollary*}{Corollary}
\newtheorem*{problem*}{Problem}
\newtheorem{theorem}[lemma]{Theorem}
\newtheorem{conjecture}[lemma]{Conjecture}
\newtheorem{corollary}[lemma]{Corollary}
\newtheorem{proposition}[lemma]{Proposition}

\theoremstyle{definition}
\newtheorem{definition}[lemma]{Definition}

\newtheorem{remark}[lemma]{Remark}

\newcommand{\F}[1]{\mathscr{#1}}
\newcommand{\fto}[1]{\stackrel{#1}{\to}}

\newcommand{\Z}{\mathbb{Z}}
\newcommand{\N}{\mathbb{N}}

\newcommand{\C}{\mathbb{C}}
\newcommand{\Q}{\mathbb{Q}}
\newcommand{\R}{\mathbb{R}}

\newcommand{\OO}{\mathcal{O}}
\newcommand{\te}{\otimes}
\newcommand{\sm}{\setminus}
\newcommand{\id}{\mathrm{id}}

\newcommand{\cF}{\mathcal F}
\newcommand{\cQ}{\mathcal Q}
\newcommand{\cA}{\mathcal A}
\newcommand{\cM}{\mathcal M}

\renewcommand{\P}{\mathbb{P}}

\DeclareMathOperator{\ch}{ch}

\DeclareMathOperator{\Hom}{Hom}

\DeclareMathOperator{\Pic}{Pic}

\DeclareMathOperator{\Eff}{Eff}

\DeclareMathOperator{\rk}{rk}

\DeclareMathOperator{\Ext}{Ext}
\DeclareMathOperator{\ev}{ev}
\DeclareMathOperator{\Coh}{Coh}
\DeclareMathOperator{\sHom}{\mathcal{H} \textit{om}}

\DeclareMathOperator{\SL}{SL}
\DeclareMathOperator{\coh}{coh}
\DeclareMathOperator{\udim}{\underline{\dim}}

\begin{document}

\date{\today}
\author{Jack Huizenga}
\address{Department of Mathematics\\University of Illinois at Chicago, Chicago, IL 60607}
\email{huizenga@math.uic.edu}
\thanks{This material is based upon work supported under a National Science Foundation Graduate Research Fellowship and a National Science Foundation Mathematical Sciences Postdoctoral Research Fellowship}
\subjclass[2010]{Primary: 14C05. Secondary: 14E30, 14J60, 13D02}

\title[Effective divisors on Hilbert schemes]{Effective divisors on the Hilbert scheme of points in the plane and interpolation for Stable bundles}

\begin{abstract}
We compute the cone of effective divisors on the Hilbert scheme of $n$ points in the projective plane.  We show the sections of many stable vector bundles satisfy a natural interpolation condition, and that these bundles always give rise to the edge of the effective cone of the Hilbert scheme.  To do this, we give a generalization of Gaeta's theorem on the resolution of the ideal sheaf of a general collection of $n$ points in the plane.  This resolution has a natural interpretation in terms of Bridgeland stability, and we observe that ideal sheaves of collections of points are destabilized by exceptional bundles.  By studying the Bridgeland stability of exceptional bundles, we also show that our computation of the effective cone of the Hilbert scheme is consistent with a conjecture in \cite{ABCH} which predicts a correspondence between Mori and Bridgeland walls for the Hilbert scheme.
\end{abstract}

\maketitle

\newcommand{\spacing}[1]{\renewcommand{\baselinestretch}{#1}\large\normalsize}
 \setcounter{tocdepth}{1}
\tableofcontents

\section{Introduction}

For a projective variety $X$, the Hilbert scheme $X^{[n]}$ parameterizes length $n$ zero-dimensional subschemes of $X$.  When $X$ is a smooth surface, the Hilbert scheme is a useful compactification of the open symmetric product $(X^n\setminus \Delta)/S_n$ parameterizing distinct collections of $n$ points.  By a result of Fogarty, in this case $X^{[n]}$ is a smooth projective variety of dimension $2n$, and $X^{[n]}\to X^n/S_n$ resolves the singularities of the ordinary symmetric product \cite{Fogarty1}. 

An interesting topic in birational geometry is to describe the various birational models of moduli or parameter spaces.  It is often the case that these models themselves admit interesting modular interpretations, and describing these alternate compactifications is of particular interest.  In \cite{ABCH}, the problem of carrying out the minimal model program for the Hilbert scheme $\P^{2[n]}$ is discussed.

A first step in carrying out the minimal model program for the Hilbert scheme $\P^{2[n]}$ is to describe the full cone of effective divisors.  In this paper, we compute the cone for every $n$, building on results from \cite{HuizengaPaper} where a partial answer was obtained.  Along the way, we will be led to consider natural interpolation questions for stable vector bundles.  We will also develop a generalization of Gaeta's theorem on resolutions of ideal sheaves of general collections of points in $\P^2$.  This generalization will illuminate many cohomological properties of such ideal sheaves.

The Picard group of $\P^{2[n]}$ has rank $2$, and is generated over $\Z$ by classes $H$ and $\Delta/2$, where $H$ is the locus of schemes meeting a fixed line and $\Delta$ is the locus of nonreduced schemes \cite{Fogarty2}.  The boundary divisor $\Delta$ always spans one edge of the cone $\Eff\P^{2[n]}$ of effective divisors (\cite{ABCH,HuizengaPaper,thesis}).  For many values of $n$ it is easy describe the other edge of this cone.  For instance, if $n = {r+2 \choose 2}$ is a triangular number, then there is a divisor given as the locus of schemes which lie on some curve of degree $r$, and this divisor spans the edge of the cone.  Alternately, the divisor is described as the locus of schemes which fail to impose independent conditions on sections of the line bundle $\OO_{\P^2}(r)$.  Generalizing this construction to allow more arbitrary vector bundles instead of only line bundles allows us to construct the nontrivial edge of $\Eff \P^{2[n]}$ for every $n$.
 
\subsection{Interpolation for vector bundles} Let $E$ be a vector bundle of rank $r$ on a smooth curve or surface $X$.  We say $E$ \emph{satisfies interpolation for $n$ points} if a general collection $Z$ of $n$ points imposes $rn$ conditions on sections of $E$, i.e. if $$h^0(E\te I_Z) = h^0(E)-rn.$$ This forces $h^0(E)\geq rn$; let $W\subset H^0(E)$ be a general subspace of dimension $rn$.  Then the locus of $Z\in X^{[n]}$ which fail to impose independent conditions on sections in $W$ forms an effective divisor $D_E(n)$ in $X^{[n]}$.

In the particular case where $X=\P^{2}$, one sees by an elementary Grothendieck-Riemann-Roch calculation (see \cite{ABCH,HuizengaPaper,thesis}) that if $E$ is a vector bundle which satisfies interpolation for $n$ points then the divisor $D_E(n)$ has class $$[D_E(n)]=c_1(E)-\rk(E)\frac{\Delta}{2}.$$ If we let $\mu(E) = c_1(E)/\rk(E)$ be the \emph{slope}, then this class spans the ray $$\mu(E)-\frac{1}{2}\Delta.$$ We are thus led to determine the minimum possible slope $\mu$ of a vector bundle satisfying interpolation for $n$ points.  We will see that the extremal edge of the effective cone can always be realized as a divisor associated to a vector bundle in this way.  Furthermore, the minimum slope $\mu$ and vector bundles of slope $\mu$ with interpolation can be explicitly described.

\subsection{Stable vector bundles and interpolation}
The key to determining the minimum slope $\mu$ of a vector bundle with interpolation for $n$ points lies in considering \emph{stable} vector bundles.  A vector bundle $E$ is \emph{(slope)-stable} if every coherent subsheaf $F\subset E$ with $0<\rk F< \rk E$ has $\mu(F)<\mu(E)$.  Our general expectation is that a stable bundle $E$ of rank $r$ typically satisfies interpolation for $n$ points so long as it has at least $rn$ sections.  The number of sections of a general stable bundle $E$ of slope $\mu(E)\geq 0$ is just $\chi = \chi(E)$ by \cite{GottscheHirschowitz}, so it is natural to try and determine the minimum possible slope of a stable bundle with the property $\chi\geq rn$.  A priori it could happen that no such minimum exists, since the infimum of the slopes of such bundles could be irrational.  This, however, is not the case.

\begin{theorem}\label{minExistTheorem}
For a fixed nonnegative integer $n$, the set of nonnegative slopes of stable bundles on $\P^2$ satisfying $\chi/r \geq n$ has a minimum $\mu$.
\end{theorem}

The number $\mu$ can be explicitly computed for any given $n$, as we will see in Section \ref{gammaSection}.  The main difficulty is understanding when a given moduli space $M(r,c,d)$ of stable vector bundles with Chern character $(r,c,d)$ is nonempty.  The answer to this question is well-known, but depends intimately on the geometry of the \emph{exceptional bundles} on $\P^2$; these are the rigid stable bundles, i.e. the stable bundles $E$ with $\Ext^1(E,E)=0$.

It is particularly important to understand the set $\F E$ of slopes of exceptional bundles on $\P^2$.  By results of Drezet and Le Potier \cite{Drezet,DLP,LePotierLectures}, for any rational number $\mu\in \Q$ there is an \emph{associated exceptional slope} $\alpha\in \F E$ such that the existence of stable vector bundles with slope $\mu$ is controlled by the (unique) exceptional bundle $E_\alpha$ of slope $\alpha$.  Theorem \ref{minExistTheorem} follows from a study of the number theory of exceptional slopes.  These slopes have many interesting properties; we highlight one result in this direction here.
 
 \begin{theorem}
 Let $\alpha\in \F E$ be the slope of an exceptional bundle on $\P^2$.  Every term in the even-length continued fraction expansion of the fractional part of $\alpha$ is a one or a two.  Furthermore, these terms form a palindrome.
 \end{theorem}

\subsection{A generalization of Gaeta's theorem}
Let us recall Gaeta's theorem on the resolution of the ideal sheaf of $n$ general points in $\P^2$ (see \cite{Eisenbud}).  Let $Z$ be a general collection of $n$ points in $\P^2$, and write $$n = \frac{r(r+1)}{2}+s \qquad (0\leq s\leq r);$$ there is a unique such decomposition.  Then the ideal sheaf $I_Z$ admits one of the following two resolutions, depending on whether $2s\leq r$ or $2s\geq r$:
$$0\to \OO_{\P^2}(-r-1)^{r-2s} \oplus\OO_{\P^2}(-r-2)^{s}\to  \OO_{\P^2}(-r)^{r-s+1}\to I_Z\to 0\hphantom{.}$$
$$0\to \OO_{\P^2}(-r-2)^s \to \OO_{\P^2}(-r)^{r-s+1}\oplus \OO_{\P^2}(-r-1)^{2s-r} \to I_Z\to 0.$$
Let us focus on the case where $2s\geq r$; a similar picture applies in the other case.  We have $r-s+1 = \dim \Hom(\OO_{\P^2}(-r),I_Z)$, so define a sheaf $W$ to be the kernel of the canonical map $$ \OO_{\P^2}(-r)\te \Hom(\OO_{\P^2}(-r),I_Z)\to I_Z.$$ Assuming the canonical map is surjective (which it is if say $s \leq r-2$), $W$ will be a vector bundle with resolution $$0 \to W \to \OO_{\P^2}(-r-2)^s\to \OO_{\P^2}(-r-1)^{2s-r}\to 0,$$ where the map is the same one as in the resolution of $I_Z$ (and hence is general).  

The resolution $$0\to W\to \OO_{\P^2}(-r)\te \Hom(\OO_{\P^2}(-r),I_Z)\to I_Z\to 0$$ is particularly well-behaved when $W$ is \emph{stable}; however, $W$ will be stable if and only if either $\varphi^{-1}<s/r\leq 1$ or $s/r$ is a convergent in the continued fraction expansion of $\varphi^{-1}$, where $\varphi = (1+\sqrt{5})/2$ is the golden ratio.  When $W$ is not stable, this resolution is somewhat unsatisfactory.  In this case the terms of the Gaeta resolution do not optimally reflect cohomological properties of the ideal sheaf $I_Z$; in particular, it may very well happen that $V\te I_Z$ has no cohomology for some vector bundle $V$, but that $V\te W$ and $V(-r)$ have cohomology.  If one wishes to use the Gaeta resolution to prove $V \te I_Z$ has no cohomology, it becomes necessary to analyze the maps in the resolution, and things become unwieldy.  Our generalization of Gaeta's theorem will resolve the ideal sheaf into a pair of semistable bundles which are much more suitable for such computations.

\begin{theorem}\label{GaetaResThm}
Let $\mu$ be the minimum slope of a stable bundle on $\P^2$ satisfying $\chi(E)/\rk(E) = n$.  Let $\alpha\in \F E$ be the exceptional slope associated to $\mu$.  If $\mu < \alpha$, the general ideal sheaf $I_Z$ of $n$ points admits a canonical resolution $$0\to W \to E_{-\alpha}\te \Hom(E_{-\alpha},I_Z)\to I_Z\to 0$$ where $W$ is a stable bundle and $E_{-\alpha}$ is the exceptional bundle of slope $-\alpha$.  Similarly, if $\mu>\alpha$, the general ideal sheaf admits a canonical resolution $$0\to E_{-\alpha-3}\te \Ext^1(I_Z,E_{-\alpha-3})^*\to W\to I_Z\to 0,$$ where $W$ is a stable bundle.  In each case, a resolution of $W$ by semi-exceptional bundles can be explicitly described.  (For some sporadic values of $n$, $W$ is actually an object of a derived category; see Theorem \ref{resTheorem} for a precise statement.)
\end{theorem}

We also refer the reader to Theorem \ref{resTheorem} for a statement in case $\mu = \alpha$.  We note that the Gaeta resolution is recovered as the special case where the exceptional slope $\alpha$ is an integer.

\subsection{The effective cone of $\P^{2[n]}$}  By combining Theorem \ref{GaetaResThm} with the main result from \cite{HuizengaPaper} we can construct the extremal edge of the effective cone of $\P^{2[n]}$.
\begin{theorem}\label{effConeThm1}
Let $\mu$ be the minimum slope of a stable bundle on $\P^2$ with $\chi/r=n$.  A general such bundle $V$ with sufficiently large and divisible rank satisfies interpolation for $n$ points.  Thus $\mu H -\frac{1}{2}\Delta$ is the class of an effective divisor on $\P^{2[n]}$.  Furthermore, the effective cone of $\P^{2[n]}$ is spanned by $$\mu H - \frac{1}{2}\Delta \qquad \textrm{and} \qquad \Delta.$$
\end{theorem}

Given the resolution of the ideal sheaf $I_Z$, showing that such bundles $V$ satisfy interpolation amounts to a previously studied problem about orthogonality of representations of the Kronecker quiver with two vertices and $N$ arrows.  We use results of Schofield and van den Bergh to prove interpolation holds.

The computational value of the generalized Gaeta resolution is demonstrated by the fact that if say $\mu<\alpha$ (as in Theorem \ref{GaetaResThm}) and $V$ is a bundle as in Theorem \ref{effConeThm1}, then both $V\te W$ and $V\te E_{-\alpha}$ turn out to have no cohomology.  Thus there is no need to understand the map $W\to E_{-\alpha}\te \Hom(E_{-\alpha},I_Z)$ to show $V\te I_Z$ has no cohomology.  Such complication was unavoidable with the original Gaeta resolution.

To show that the divisor $\mu H - \frac{1}{2}\Delta$ is actually extremal, we will study the rational map $\P^{2[n]}\dashrightarrow M(\ch(W))$ sending the general scheme $Z$ to the bundle $W$ in the resolution of $I_Z$ (or, more precisely, we study the map to a related moduli space of quiver representations).  This map typically has positive-dimensional fibers, and our extremal divisors on the Hilbert scheme are pullbacks under this map.

Since it is a bit tedious to determine the exact value of $\mu$ in the theorem by hand, we give a table describing the effective cone and associated exceptional slopes for small $n$ at the end of Section \ref{effConeSection}.

\subsection{Bridgeland stability}  The generalized Gaeta resolution has further relevance when one discusses the Bridgeland stability of ideal sheaves $I_Z$.  In Sections \ref{bridgelandSection} and \ref{bridgelandSec2} we will show that our computation of the effective cone $\Eff \P^{2[n]}$ is consistent with a conjecture in \cite{ABCH} predicting a correspondence between the Mori walls for $\P^{2[n]}$ and the Bridgeland walls in a suitable half-plane of stability conditions.  We will see that our resolution shows that general ideal sheaves are always destabilized by certain exceptional bundles.  The main step in the proof consists of determining when exceptional bundles are Bridgeland stable.  We give a fairly complete answer to this question in Section \ref{bridgelandSec2}.

\subsection{Further work}  It appears that many of the results in this paper can be generalized to study cones of divisors on moduli spaces of semistable sheaves on the plane.  The Picard group of a moduli space $M(\xi)$ of semistable sheaves is naturally identified with a plane of orthogonal Chern characters \cite{LePotierLectures}.  In terms of this description, the effective cone should correspond to Chern characters of stable orthogonal bundles.  We will study this problem in upcoming work with Izzet Coskun and Matthew Woolf.

\subsection{Acknowledgements}  I would like to thank Joe Harris, Izzet Coskun, and Daniele Arcara for their many helpful discussions regarding this work.  Also, I am indebted to the anonymous referee of \cite{HuizengaPaper}, whose suggestions pointed me towards the methods used in this paper.  The referees of this article also provided very valuable advice, helping to greatly simplify the material from Sections \ref{SteinerSection} and \ref{effConeSection}. I would also like to thank So Okada and RIMS Kyoto for organizing a very useful conference on related topics.  Finally, I would like to thank the community of MathOverflow for their help with finding references for basic facts regarding continued fractions.

\section{Preliminaries}

In this section we set notation for the paper and review parts of the classification of stable vector bundles on $\P^2$ that will be necessary throughout the paper.  We predominantly choose notations to agree with the papers of Drezet and Le Potier \cite{DrezetBeilinson,Drezet,DLP,LePotierLectures}, and summarize results from those sources.

\subsection{Invariants of coherent sheaves} 

We collect here several formulas which will be used constantly throughout the paper.  Let $E$ be a coherent sheaf on $\P^2$, with Chern character $(\ch_0,\ch_1,\ch_2)=(r,c_1,\ch_2)$. When $r>0$, the \emph{slope} and \emph{discriminant} are defined by 
$$\mu(E) = \frac{c_1}{r} \qquad \textrm{and} \qquad \Delta(E) = \frac{1}{2}\mu^2-\frac{\ch_2}{r},$$ respectively.  The Riemann-Roch formula relates the Chern character to the Euler characteristic by $$\chi(E) = r(P(\mu)-\Delta),$$ where $$P(x)=\frac{1}{2}(x^2+3x+2)$$ is the Hilbert polynomial of the trivial sheaf $\OO_{\P^2}$.  

If $F$ is another coherent sheaf, we put $$\chi(E,F) = \sum_{i=0}^2 (-1)^i \dim \Ext^i(E,F).$$ In case both $E$ and $F$ have positive rank, a variant of Riemann-Roch shows $$\chi(E,F) = r(E)r(F)(P(\mu(F)-\mu(E))-\Delta(E)-\Delta(F)).$$ Finally, Serre duality for $\Ext$-groups gives $$\Ext^i(E,F) \cong \Ext^{2-i} (F,E(-3))^{*}$$ for each $i$ \cite[Proposition 1.2]{DLP}.

We say that a sheaf $E$ is \emph{acyclic} if $H^i(E)=0$ for all $i>0$.  In practice, we will only consider the notion of acyclicity for sheaves $E$ with $\chi(E) = 0$, in which case $H^0(E) = 0$ as well.

\subsection{Exceptional bundles}  The sources \cite{DLP,LePotierLectures} are good references for the material in this subsection.  A coherent sheaf $E$ is said to be \emph{stable} (resp.  \emph{semi-stable}) if it is torsion free and every coherent subsheaf $F\subset E$ with $0<r(F)<r(E)$ has $\mu(F)\leq \mu(E)$ with $\Delta(E)<\Delta(F)$ (resp. $\leq$) in case of equality.  For fixed values of the Chern character $\ch = (\ch_0,\ch_1,\ch_2)$, we denote by $M(\ch)$ the moduli space of semistable sheaves with $\ch(E) = \ch$.  

The invariant $\Delta$ is useful due to its connection with stable bundles.  Bogomolov's theorem shows that $\Delta(E)\geq 0$ for a stable bundle $E$.  Furthermore, when the moduli space $M(\ch)$ is nonempty, it is irreducible of dimension $r^2(2\Delta-1)+1$.  In particular, if $M(\ch)$ consists of a single point, then $\Delta<1/2$.

An \emph{exceptional bundle} $E$ is a stable coherent sheaf such that $M(\ch(E))$ is reduced to a point (it follows from this that $E$ is homogeneous, hence locally free).  Equivalently, it is a stable bundle with $\Delta(E)<1/2$, or a rigid stable bundle (i.e. a stable bundle with $\Ext^1(E,E)=0$).  A \emph{semi-exceptional bundle} is a bundle of the form $E^k$, with $E$ exceptional.

For any rational number $\alpha\in \Q$, denote by $r_\alpha$ the denominator of $\alpha$, i.e. the smallest positive integer $r>0$ with $r\alpha\in \Z$.  If there exists an exceptional bundle $E_\alpha$ of slope $\alpha$, then it is unique and its invariants are given by 
$$\rk(E_\alpha) = r_\alpha \qquad c_1(E_\alpha) = \alpha r_\alpha \qquad \Delta_\alpha := \Delta(E_\alpha) = \frac{1}{2}\left(1-\frac{1}{r_\alpha^2}\right) \qquad \chi_\alpha := \chi(E_\alpha) = r_\alpha(P(\alpha)-\Delta_\alpha).$$ Note that $E_\alpha^* = E_{-\alpha}$, and also $E_\alpha(1) = E_{\alpha+1}$.

Exceptional bundles play an important role in the problem of determining when the moduli spaces $M(\ch)$ are nonempty.  In particular, it is necessary to understand the set $\F E$ of slopes of exceptional bundles.

Clearly $\F E$ is invariant under translation $\alpha \mapsto \alpha+1$ and inversion $\alpha\mapsto -\alpha$.  If $\alpha,\beta\in \F E$ and $3+\alpha-\beta\neq 0$, we define a rational number $$\alpha.\beta=\frac{\alpha+\beta}{2} + \frac{\Delta_\beta-\Delta_\alpha}{3+\alpha-\beta},$$ which should be thought of as a modification of the mean of $\alpha$ and $\beta$ (note that there is a typo in \cite{Drezet}, and that $\Delta_\beta$ and $\Delta_\alpha$ are reversed there). Let $\F D=\Z[\frac{1}{2}]$ be the set of dyadic rational numbers.  There is a bijection $\varepsilon: \F D \to \F E$ described inductively by setting $\varepsilon(n) = n$ for $n\in \Z$ and 

\renewcommand{\arraystretch}{1.3}

$$
\varepsilon\left(\frac{2p+1}{2^q}\right) = \varepsilon\left(\frac{p\vphantom{2}}{2^{q-1}}\right).\varepsilon\left(\frac{p+1}{2^{q-1}}\right).$$ It is useful to keep several values of $\varepsilon$ where $q$ is small in mind, so we record them here.
$$\begin{array}{c|ccccccccc} \frac{p}{2^q}& 0 & \frac{1}{8}& \frac{1}{4} &\frac{3}{8} & \frac{1}{2} & \frac 58 & \frac 34 & \frac 78 & 1  \\ \hline  \varepsilon\left(\frac{p}{2^q}\right) & 0 & \frac 5{13} & \frac 25 & \frac{12}{29} & \frac 12 & \frac{17}{29}  & \frac 35 & \frac {8}{13} & 1
 \end{array}$$ The following arithmetic properties of this setup will be used repeatedly in computations.  

\begin{lemma}\label{numericalProps}
Suppose $$\alpha = \varepsilon\left(\frac{p\vphantom{1}}{2^q}\right) \qquad \beta = \varepsilon\left(\frac{p+1}{2^q}\right).$$ Then 
\begin{enumerate}
\item $\alpha < \alpha.\beta < \beta$,
\item $r_{\alpha.\beta} = r_\alpha r_\beta ( 3 - \alpha+\beta)$, and
\item $P(\alpha-\beta) = \Delta_\alpha+\Delta_\beta$.   
\end{enumerate}
Furthermore, the relations $$\alpha.\beta - \alpha = \frac{1}{r_\alpha^2(3+\alpha-\beta)} \qquad \textrm{and} \qquad \beta - \alpha.\beta = \frac{1}{r_\beta^2(3+\alpha-\beta)}$$ hold.
\end{lemma}

The ``furthermore'' part of the lemma is an elementary consequence of the previous properties.  Most properties of exceptional slopes are more efficiently proved by using the identities in the lemma instead of invoking the explicit definition of $\alpha.\beta$.

\subsection{Existence of stable coherent sheaves on $\mathbb{P}^2$} For any $\alpha\in \F E$, define a number $$x_\alpha = \frac{3}{2}-\sqrt{\frac{9}{4}-\frac{1}{r_\alpha^2}},$$ which is the smaller of the two solutions of the equation $P(-x)-\Delta_\alpha = \frac{1}{2}$.   The number $x_\alpha$ is always irrational.  We denote by $I_\alpha\subset \R$ the interval $$I_\alpha = (\alpha-x_\alpha,\alpha+x_\alpha).$$ The intervals $I_\alpha$ are all disjoint, and they cover the rationals:  $$\Q =\Q \cap \bigcup_{\alpha\in \F E} I_\alpha.$$ If $\mu\in \Q$, then the unique slope $\alpha\in \F E$ with $\mu\in I_\alpha$ is called the \emph{associated exceptional slope} to $\mu$.

\begin{theorem}[Drezet \cite{Drezet}]\label{deltaClassification}
Suppose $r\geq 1$ is an integer, and $\mu,\Delta\in \Q$ are numbers such that $r\mu$ and $r(P(\mu)-\Delta)$ are integers.  Define a function $\delta:\Q\to \Q$ by the formula $$\delta(\mu) = P(-|\mu-\alpha|)-\Delta_\alpha \qquad \textrm{if $\mu\in I_\alpha$}$$ The moduli space $M(r,\mu,\Delta)$ of semistable sheaves  with invariants $(r,\mu,\Delta)$ is nonempty if and only if either $$\delta(\mu)\leq \Delta$$ or $(r,\mu,\Delta)$ are the invariants of some semi-exceptional bundle.
\end{theorem}

Write $C = \R\setminus \bigcup_{\alpha\in \F E} I_\alpha$.  We can view $C$ as a generalized Cantor set, obtained by iteratively removing from $\R$ at step $q$ all intervals $I_\alpha$ where $\alpha$ is of the form $\varepsilon(p/2^q)$.  It is easy to see from what has been said so far that $C$ is the closure of all the endpoints of the intervals $I_\alpha$ (just as is true for the ordinary Cantor set).  

\begin{remark}
As with the standard Cantor set, $C$ is uncountable and most of its points are not endpoints of the intervals $I_\alpha$.  This fact is a source of much technical difficulty.
\end{remark}

While the next result is well-known, the argument is fundamental to our discussion, so we include it.

\begin{proposition}\label{deltaContinuous}
The function $\delta:\Q\to\Q$ admits a unique continuous extension to a function $\R\to \R$, and $\delta^{-1}(1/2) = C$.
\end{proposition}
\begin{proof}
We can define $\delta$ on each interval $I_\alpha$ by the formula $$\delta(\mu) = P(-|\mu-\alpha|)-\Delta_\alpha,$$ so it is clear that this extension of $\delta$ is continuous everywhere except the points in $C$, where it has not yet been defined.  Noting that $$\lim_{\mu \to (\alpha+x_\alpha)^-} \delta(\mu) = \lim_{\mu\to (\alpha-x_\alpha)^+} \delta(\mu) = \frac{1}{2}$$ by the definition of $x_\alpha$, we see that any continuous extension of $\delta$ to $\R$ must satisfy $\delta(\xi) = 1/2$ for all $\xi \in C$.  Thus we define $\delta(\xi)=1/2$ for all $\xi\in C$.  We also observe that $\delta(\mu) > 1/2$ for all $\mu\in \R\setminus C$ since $\delta$ is increasing on each interval $(\alpha-x_\alpha,\alpha]$ and decreasing on each interval $[\alpha,\alpha+x_\alpha)$.  We must show continuity holds at $\xi\in C$.  

If $\xi\in C$ is of the form $\alpha+x_\alpha$, then clearly $\delta$ is left-continuous at $\xi$.  Similarly, if $\xi$ is of the form $\beta-x_\beta$, it is right-continuous there.  Without loss of generality, suppose $\xi$ is not of the form $\alpha+x_\alpha$; we show $\delta$ is left-continuous at $\xi$.  Since $\xi$ is in $C$ but not of the form $\alpha+x_\alpha$, it is an increasing limit of exceptional slopes.  If $\alpha\in \F E$ is any exceptional slope with $\alpha<\xi$, then the maximum value of $\delta$ on the interval $I_\alpha$ occurs at $\alpha$, and equals $$\delta(\alpha) = \frac{1}{2}+\frac{1}{2r_\alpha^2}.$$  For any $\epsilon>0$, we can choose an exceptional slope $\alpha<\xi$ sufficiently close to $\xi$ such that all rationals $\mu\in [\alpha,\xi)$ satisfy $(2r_\mu^2)^{-1}<\epsilon$; then for all $x\in (\alpha,\xi)$ we will have $|\delta(\xi)-\delta(x)|< \epsilon$.  
\end{proof}

\subsection{Triads; resolutions of height 0 stable sheaves on $\mathbb{P}^2$}\label{triadSection}
A \emph{triad} is a triple $(E,G,F)$ of exceptional bundles such that the slopes $(\mu(E),\mu(G),\mu(F))$ are of the form $(\alpha, \alpha.\beta,\beta)$, $(\beta-3,\alpha,\alpha.\beta)$, or $(\alpha.\beta,\beta,\alpha+3)$, where $\alpha,\beta$ are exceptional slopes of the form $$\alpha = \varepsilon\left(\frac{p\vphantom{1}}{2^q}\right) \qquad \beta = \varepsilon\left(\frac{p+1}{2^q}\right)$$ for some $p,q$  (possibly with $q=-1$, so that e.g. $(\OO_{\P^2},\OO_{\P^2}(1),\OO_{\P^2}(2))$ is a triad).  Any exceptional slope can be written in the form $\alpha.\beta$, so any exceptional bundle can be viewed as the bundle of slope $\alpha.\beta$ in any of the three types of triads.  The results from the first part of this subsection can be found in \cite{DrezetBeilinson}.

For any triad $(E,G,F)$, the canonical map $$\ev^*:G\to F\te \Hom(G,F)^*$$ is injective, and the cokernel is an exceptional bundle $S$ (for a discussion of \emph{which} exceptional bundle $S$ is, see either \cite{DrezetBeilinson} or Theorem \ref{kernelSlope} in this paper). On the other hand, the map $$\ev:E\otimes \Hom(E,G)\to G$$ is surjective, with kernel $S(-3)$.  For any coherent sheaf $V$ on $\P^2$, there is a canonical complex $$E\te \Ext^1(V,E)^\ast \fto{A_V} G\te \Ext^1(S,V)\fto{B_V} F\te \Ext^1(F,V)$$ coming from a generalized version of the Beilinson spectral sequence.  If $\Hom(F,V) = \Hom(V,E) = 0$, then the map $A_V$ is injective, the map $B_V$ is surjective, and the middle cohomology is just $V$.  

Many numerical invariants of pairs of members of a triad are easily computed, in light of the following vanishing theorem.  
\begin{theorem}[Drezet {\cite[Theorem 6]{DrezetBeilinson}}]\label{excepOrthogonalThm}
If $E,F$ are any exceptional bundles with $\mu(E)\leq \mu(F)$, then $\Ext^i(E,F)=0$ for $i>0$.
\end{theorem}
 One then easily concludes the following facts by computing Euler characteristics:
$$F^*\te E, \quad F^*\te G,\quad \textrm{and} \quad G^*\te E \quad \textrm{are acyclic}$$
$$\dim \Hom(E,G) = 3\rk(F) \qquad \dim \Hom(G,F) = 3\rk(E) \qquad \rk(S) = 3\rk(E)\rk(F)-\rk(G)$$

Let $V$ be a stable sheaf with invariants $(r,\mu,\Delta)$, and let $\alpha\in \F E$ be the exceptional slope associated to $\mu$.  The \emph{height} of $V$ is defined to be the integer $$h(V) = rr_\alpha(\Delta-\delta(\mu)).$$ In case $\mu\leq \alpha$, this is just the number $-\chi( E_\alpha,V)$; similarly, in case $\mu\geq \alpha$ it equals $-\chi(V,E_\alpha)$.  

In the case where the height is zero, the above complex degenerates considerably, as discussed in \cite{Drezet}.  To see this, suppose $V$ has height zero, and first assume $\alpha - x_\alpha < \mu\leq \alpha$.     Choose a triad $(E,G,F)$ with $F = E_\alpha$.  We have inequalities of slopes $$\mu(E)<\mu(G)<\mu(V)\leq \mu(F).$$ The height zero hypothesis gives $\chi(F,V) = 0$.  Stability and the fact that $V$ is non-exceptional gives $\Hom(F,V)=0$ and $\Ext^2(F,V) = 0$ (by Serre duality).  Thus also $\Ext^1(F,V)=0$.  Stability also gives $\Hom(V,E) = 0$, so we conclude that the complex gives an exact sequence $$0\to E\te \Ext^1(V,E)^*\to G\te \Ext^1(S,V)\to V\to 0.$$ If we write this resolution in the form $$0\to E^{m_1}\to G^{m_2}\to V\to 0,$$ then the hypothesis that $\alpha-x_\alpha < \mu \leq \alpha$ is equivalent to the inequalities \begin{equation}\label{slopeIneq} r_\alpha x_\alpha < \frac{m_1}{m_2} \leq \frac{r_\alpha}{\rk S}.\end{equation} In case $\alpha\leq \mu < \alpha-x_\alpha$, we choose a triad $(E,G,F)$ with $E = E_\alpha$, and an identical argument gives an exact sequence $$0\to V\to G\te \Ext^1(S,V)\to F\te \Ext^1(F,V)\to 0.$$ This time, writing the resolution in the form $$0\to V \to G^{m_2}\to F^{m_1}\to 0$$ the same inequalities (\ref{slopeIneq}) also hold, where it is understood that $S$ has changed because we are using a different triad.

\subsection{A Bertini-type statement} Throughout the paper, the following setup will occur several times.  Suppose $E$, $F$ are vector bundles of ranks $m,n$ on a smooth variety $X$ and the sheaf $\sHom(E,F)$ is globally generated.  For a map $\phi:E\to F$, denote by $D_k(\phi)$ the degeneracy locus $\{x\in X:\rk \phi_x\leq k\}$.

\begin{proposition}\label{bertiniProp}
With the preceding setup, if $\phi$ is general then $D_k(\phi)$ is empty or has the expected codimension $(m-k)(n-k)$.  Furthermore, in case the general $D_k(\phi)$ is nonempty, the locus of $\phi\in\Hom(E,F)$ where $D_k(\phi)$ has greater than the expected dimension is at least of codimension $2$.
\end{proposition}
\begin{proof}We quickly sketch the argument, which is just an analysis of the proof of \cite[Theorem 2.8]{Ottaviani}.  By global generation, we have a surjection $$H^0(E^*\te F)\te \OO_X\to E^*\te F\to 0,$$
which shows that the natural evaluation map $$\ev: X\times \P H^0(E^*\te F)\to \P(E^* \te F)$$ is surjective and has fibers isomorphic to $\P^{h^0(E^*\te F)-mn}$.  There is a subvariety $\Sigma_k \subset \P(E^\ast\te F)$ consisting of those points $\phi_x:E_x\to F_x$ such that $\rk(\phi_x)\leq k$, and it is irreducible of codimension $(m-k)(n-k)$.  Then $Z = \ev^{-1}(\Sigma_k)$ is an irreducible variety of dimension $h^0(E^*\te F)-(m-k)(n-k)-1.$  If the projection $q:Z\to \P H^0(E^* \te F)$ is surjective, then the general fiber has dimension $\dim X-(m-k)(n-k)$, so $D_k(\phi)$ has codimension $(m-k)(n-k)$.  Furthermore, in this case the dimension of the fibers of $q$ cannot jump in codimension $1$, as this would violate the irreducibility of $Z$.  Alternately, if $q$ is not surjective then $D_k(\phi)$ is empty for general $\phi$.
\end{proof}

\section{Number-theoretic properties of exceptional slopes}\label{numTheorySection}

The exceptional slopes $\alpha\in \F E$ have many surprising number-theoretic properties.  These will be of utmost importance in proving that the set of nonnegative slopes of stable bundles $V$ such that $\chi(V)/\rk(V) \geq q,$ for $q\in \Q_{\geq 0}$ a fixed nonnegative rational, has a minimum.  

The main goal of this section is to describe nice properties of the continued fraction expansion of any $\alpha\in \F E$.  To do this, we essentially give an algorithm which computes the continued fraction expansion of $\alpha$ in terms of the binary expansion of the dyadic number $p/2^q$ with $\varepsilon(p/2^q) = \alpha$.  

Since the set $\F E$ of exceptional slopes is invariant under translation by $1$, it will suffice to consider only the case where $0\leq \alpha <1$.
For any real numbers $a_0,\ldots,a_k$ for which it makes sense, define the number $$[a_0;a_1,\ldots,a_k] := a_0 + \cfrac{1}{a_1+
  \cfrac{1}{\ddots \raisebox{-1.2ex}{${}+\cfrac{1}{a_k}$}}}
 $$ Recall that any rational number $0\leq \alpha < 1$ has a unique continued fraction expansion $\alpha = [0;a_1,\ldots,a_k]$
 where the $a_i$ are positive integers and $k$ is even.  Indeed, if $k$ is odd with $a_k=1$ then we can write $\alpha = [0;a_1,\ldots,a_{k-1}+1]$; on the other hand if $k$ is odd and $a_k>1$ then $\alpha = [0;a_1,\ldots,a_{k}-1,1].$

Following standard notation, we let $p_n$ and $q_n$ be the numerator and denominator of the rational number $[0;a_1,\ldots,a_n]$, called the $n$th \emph{convergent} of $\alpha$.  With this notation, $\alpha = p_k/q_k$.  The fundamental relation between convergents is encapsulated by the equality of matrices $$ \begin{pmatrix}q_n & q_{n-1}\\ p_n & p_{n-1}\end{pmatrix} = \begin{pmatrix}a_1 & 1\\ 1 & 0 \end{pmatrix}\begin{pmatrix}a_2 & 1\\ 1 & 0\end{pmatrix}\cdots \begin{pmatrix}a_n & 1\\ 1 & 0\end{pmatrix}. $$ It is immediate from computing determinants that $q_np_{n-1}-q_{n-1}p_n=(-1)^{n}$.  

We say that the continued fraction expansion of $0\leq \alpha <1$ is \emph{palindromic} if the word $a_1,a_2,\ldots,a_k$ is a palindrome, i.e. if $a_i = a_{k+1-i}$ for each $i$.  Taking transposes of the above equality of matrices and using the uniqueness of continued fraction expansions of a given length, we  recover the following well-known fact.\footnote{See \cite{MO2} for another argument.}
 
\begin{lemma}
A continued fraction expansion $[0;a_1,\ldots,a_k]$ for the number $\alpha$ is palindromic if and only if $p_k = q_{k-1}$.  That is, the denominator of the penultimate convergent equals the numerator of $\alpha$.
\end{lemma}

With preliminaries out of the way, we are now ready to state and prove our main result on the continued fraction expansion of an exceptional slope $\alpha\in \F E$.

\begin{theorem}
Let $0 \leq \alpha <1$ be an exceptional slope.  The unique continued fraction expansion $\alpha = [0;a_1,\ldots,a_k]$ with $k$ even is palindromic, and every $a_i$ is either $1$ or $2$.  Furthermore, 
\begin{enumerate}
\item every block of ones in the word $a_1,\ldots,a_k$ has even length, and
\item every block of twos in the word $a_2,\ldots,a_{k-1}$ has even length.
\end{enumerate}
\end{theorem}
\begin{proof}
The theorem is clearly true for $\alpha = 0$.  Any exceptional slope in the interval $(0,1)$ can be written uniquely in the form $\alpha.\beta$,
where $$\alpha = \varepsilon\left(\frac{p\vphantom{1}}{2^q}\right) \qquad \beta = \varepsilon\left(\frac{p+1}{2^q}\right)$$ for integers $p,q$ with $0\leq p \leq  2^q - 1$ and $q\geq 0$.  We wish to induct on $q$.  There is a slight difficulty, in that perhaps $\beta = 1$, where the integer part of the even length continued fraction expansion is not $0$.  To circumvent this, we simply note that for every $k\geq 0$ we have $$\varepsilon\left(1-2^{-k}\right)=\frac{F_{2k}}{F_{2k+1}} =[0;\underbrace{1,\ldots,1}_{2k \textrm{ copies}}],$$ where $F_0 = 0$, $F_1 = 1$, $F_{n+2} = F_{n+1}+F_n$ is the Fibonacci sequence.  Thus the theorem is true for all slopes of the form $\varepsilon(1-2^{-k}).$ For any $\alpha.\beta$ not of the form $\varepsilon(1-2^{-k})$ we will have $\beta <1$, and thus we may assume by induction that the theorem holds for $\alpha$ and $\beta$.

We now describe how to compute the continued fraction expansion of $\alpha.\beta$ when the theorem holds for $\alpha = [0;a_1,\ldots,a_k]$ and $\beta = [0;b_1,\ldots,b_l]$ (where $k,l$ are even so that the expansions are palindromic).  First suppose that $\beta$ is of the form $$\beta = [0;b_1,b_2,\ldots,b_{l-1},2].$$  Consider the continued fraction $$[0;b_1,b_2,\ldots,b_{l-1},1,1,2,a_1,\ldots,a_k].$$ Visibly every term is a $1$ or a $2$, and one easily checks the hypotheses on the lengths of blocks of ones and twos are satisfied by this new word.  The length of the word is also even.  Furthermore, the equality $[1;1,2+x]=[2;-(3+x)],$ valid for any real $x\neq -2,-3$, shows that this continued fraction equals $$[0;b_1,\ldots,b_{l-1},2,-(3+\alpha)]=[0;b_1,\ldots,b_l,-(3+\alpha)].$$ We'll show in a minute that this number is precisely $\alpha.\beta$.  Before that, we handle the other possible form of $\beta$.  If instead $$\beta = [0;b_1,\ldots,b_{l-2},1,1]$$ (recalling that ones occur in blocks of even length) we consider the fraction $$[0;b_1,\ldots,b_{l-2},2,2,a_1,\ldots,a_k],$$ again easily verifying that the condition on the parity of lengths of blocks is satisfied.  Here the equality $[2;2+x] = [1,1,-(3+x)],$ valid again for real $x\neq -2,-3$, shows that this fraction equals $$[0;b_1,\ldots,b_{l-2},1,1,-(3+\alpha)]=[0;b_1,\ldots,b_l,-(3+\alpha)].$$
Thus, in either case we must show $$\alpha.\beta = [0;b_1,\ldots,b_l,-(3+\alpha)].$$

Proving this relies on the palindromic property of the continued fraction expansion $\beta = [0;b_1,\ldots,b_l]$.  Writing $p_i/q_i$ for the convergents of $[0;b_1,\ldots,b_l,-(3+\alpha)],$ we have the relation $$\frac{p_{l+1}}{q_{l+1}} = \frac{-(3+\alpha)p_l+p_{l-1}}{-(3+\alpha)q_l+q_{l-1}},$$ and we must show this number equals $\alpha.\beta$.  From the palindromic property of $\beta$, we get $q_{l-1} = p_l$.  Writing everything in terms of $\beta$ and $r_\beta$, we thus have $$p_l = \beta r_\beta \qquad q_l = r_\beta \qquad q_{l-1} = \beta r_\beta \qquad p_{l-1} = \frac{1}{r_\beta}+\beta^2 r_\beta,$$ where $p_{l-1}$ is determined by the relation $p_{l-1}q_{l}-p_{l}q_{l-1} = 1$, recalling that $l$ is even.  Making the substitutions, basic algebra (using no special properties of $\alpha,\beta$) shows $$\beta - \frac{p_{l+1}}{q_{l+1}} = \frac{1}{r_\beta^2(3+\alpha-\beta)}.$$ Comparing this with Lemma \ref{numericalProps}, we conclude $p_{l+1}/q_{l+1} = \alpha.\beta$.

To complete the proof, it remains to show that our discovered continued fraction expansion for $\alpha.\beta$ is palindromic.  To do this, we take the expansion we found and verify that reversing the terms gives a fraction that also equals $\alpha.\beta$; by uniqueness of even length expansions we conclude the expansion is palindromic.

So first suppose we are in the case where $b_l = 2$.  The fraction obtained by reversing the terms of $$[0;b_1,\ldots,b_{l-1},1,1,2,a_1,\ldots,a_k]$$ is the fraction $$[0;a_1,\ldots,a_k,2,1,1,b_{2},\ldots,b_l]$$ making use of the palindromic hypothesis on $\alpha$ and $\beta$.  Now $[1;b_2,\ldots,b_l]=\beta^{-1}-1=[-1;\beta],$ and for any $x\neq 0,1$ we have $[2;1,-1,x]=3-x.$  Thus this fraction equals $$[0;a_1,\ldots,a_k,3-\beta].$$ Similarly, in case $b_l = b_{l-1} = 1$, the fraction obtained by reversing the terms of $$[0;b_1,\ldots,b_{l-2},2,2,a_1,\ldots,a_k]$$ is $$[0;a_1,\ldots,a_k,2,2,b_3,\ldots,b_{l}].$$ One easily checks $[2;2,b_3,\ldots,b_l]=3-\beta$, so in this case the fraction also equals $[0;a_1,\ldots,a_k,3-\beta]$.  To complete the proof we must verify that $\alpha.\beta = [0;a_1,\ldots,a_k,3-\beta]$.  Letting $p_i/q_i$ be the convergents, we use the palindromic property of $\alpha$ to easily compute $$\frac{p_{k+1}}{q_{k+1}} - \alpha = \frac{1}{r_\alpha^2(3+\alpha-\beta)}.$$ Again comparing with Lemma \ref{numericalProps}, we conclude $p_{k+1}/q_{k+1} = \alpha.\beta$.
\end{proof}

An immediate consequence of our description of the continued fraction expansion of $\alpha\in \F E$ is the following elementary congruence, which we will need later.  We do not know of a simple proof of this fact that does not make use of continued fraction methods.

\begin{corollary}\label{congruentCor}
If $\alpha\in \F E$ is an exceptional slope, then $(\alpha r_\alpha)^2 \equiv -1 \pmod{r_\alpha}$.
\end{corollary}
\begin{proof}
Clearly the congruence only depends on the fractional part of $\alpha$, so we may assume $0\leq \alpha< 1$.  If $[0;a_1,\ldots,a_k]$ is the even length palindromic continued fraction expansion of $\alpha$, then $q_k p_{k-1}- q_{k-1}p_k = 1$, which in light of the palindrome condition gives $r_\alpha p_{k-1}-(\alpha r_\alpha)^2 = 1$.
\end{proof}

\begin{corollary}\label{cantorFractionsCor}
Let $C = \R \setminus \bigcup_{\alpha\in \F E} I_\alpha$.  If $\xi \in C$, then the fractional part of the continued fraction expansion of $\xi$ has only ones and twos in it.
\end{corollary}
\begin{proof}
Since $\bigcup_{\alpha} I_\alpha$ covers the rationals, $\xi$ is irrational. Inspection of the continued fraction algorithm reveals that for each $k$, there is an $\epsilon>0$ such that all $\lambda\in (\xi-\epsilon,\xi+\epsilon)$ have the same first $k$ terms in their continued fraction expansions as $\xi$.  But every element of $C$ is a limit of exceptional slopes.\end{proof}

\begin{corollary}\label{repeatingFracCor}
Let $D>5$ be a rational number.  The number $$\xi = \frac{-3+\sqrt{D}}{2}$$ lies in $I_\alpha$ for some $\alpha\in \F E$.  That is, it is not in the generalized Cantor set $C$.
\end{corollary}

Note that the numbers $\alpha\pm x_\alpha$ are all quadratic irrationals, and they all lie in $C$.  Thus many quadratic irrationals do not lie in some $I_\alpha$.  Furthermore, when $D=5$ the result is false, as then $\xi = 0 - x_0$.  We thank Henry Cohn for showing us the following argument \cite{MO3}. 

\begin{proof}
 If $D$ is a square the result is obvious, so we may assume $\xi$ is irrational.  Also, if $5<D<9$ then $0-x_0< \xi < 0$, so $\xi\in I_0$.  Thus we assume $D>9$ is not a square.

Let $r$ be the positive integer such that $$\label{rineq}(2r+1)^2 < D < (2r+3)^2.$$ We claim the (repeating) continued fraction expansion of $\xi$ takes the form $$[r-1;\overline{a_1,\ldots,a_k}]$$ with $a_k = 2r+1$.  Since $2r+1\geq 3$, we will conclude from Corollary \ref{cantorFractionsCor} that $\xi \notin C$ and hence $\xi\in I_\alpha$ for some $\alpha$.

Consider the number $\xi + r + 2$.  The integer part of this number is $2r+1$, so it will suffice to show that it is purely periodic, i.e. that $$\xi + r + 2= [\overline{2r+1;a_1,\ldots,a_{k-1}}].$$ It is well known \cite{Davenport} that a quadratic irrational is purely periodic if and only if it is larger than $1$ and its algebraic conjugate is between $-1$ and $0$.  Thus we must only check the inequalities $$-1< r+2 + \frac{-3 - \sqrt{D}}{2}< 0,$$ and these are equivalent to $(2r+1)^2<D<(2r+3)^2$.
\end{proof}

\section{The associated exceptional bundle}\label{gammaSection}

Let $Z\in \P^{2[n]}$ be a general point.  In the next section, we will determine a particularly nice resolution of the ideal sheaf $I_Z$ by a semi-exceptional bundle and a stable bundle.  One of the more challenging aspects of finding this resolution is simply determining which exceptional bundle is the correct one.  The goal of this section is to determine the slope of the exceptional bundle which is naturally associated to the ideal sheaf $I_Z$.

As a first goal, we aim to determine the minimum possible slope $\mu$ of a stable bundle $V$ with the property that $\chi(V) \geq \rk(V)n$.  
To do this, we introduce an auxiliary function $\gamma:\Q_{\geq 0}\to \Q_{\geq 0}$ on the nonnegative rationals by the formula $$\gamma(\mu)= P(\mu)-\delta(\mu),$$ noting that $\gamma(0)=0$. By Theorem \ref{deltaClassification}, for any rational numbers $\mu,\Delta$, there exists a non-exceptional stable bundle $V$ with slope $\mu$ and discriminant $\Delta$ if and only if $\Delta \geq \delta(\mu)$.  Equivalently, there exists such a bundle if and only if $$\gamma(\mu) \geq P(\mu)-\Delta = \frac{\chi(V)}{\rk(V)}.$$  Thus for each $\mu\in \Q_{\geq 0}$ the maximum value of the ratio $\chi(V)/\rk(V)$ over all non-exceptional stable bundles $V$ of slope $\mu$ is precisely $\gamma(\mu)$.

Since $\delta$ admits a unique continuous extension to $\R$ by Proposition \ref{deltaContinuous}, we see immediately that $\gamma$ also admits a continuous extension to a function $\gamma: \R_{\geq 0} \to \R_{\geq 0}$.  Let us establish several other elementary properties of the function $\gamma$.

\begin{proposition}\label{gammaProp}
The function $\gamma:\R_{\geq 0}\to \R_{\geq 0}$
\begin{enumerate}
\item is strictly increasing,
\item is piecewise linear with rational coefficients on each interval $I_\alpha$, where $\alpha\in \F E$, and
\item is unbounded.
\end{enumerate}
In particular, $\gamma$ has an inverse.
\end{proposition}
\begin{proof}
We first record a more explicit formula for $\gamma$.  On any interval $I_\alpha$, $\gamma$ takes the form
$$\gamma(\mu) = \begin{cases} \alpha (\mu+3) +1 + \Delta_\alpha-P(\alpha) & \textrm{if }\mu\in (\alpha-x_\alpha,\alpha]\\ (\alpha+3)\mu+1+\Delta_\alpha-P(\alpha) & \textrm{if }\mu\in [\alpha,\alpha+x_\alpha).\end{cases}$$
Properties (2) and (3) follow immediately.  Clearly $\gamma$ is increasing on each interval $(\alpha-x_\alpha,\alpha+x_\alpha)$.  Thus to see $\gamma$ is strictly increasing, it will suffice to see that if $0\leq \alpha <\beta$ are exceptional slopes then $\gamma(\alpha+x_\alpha)<\gamma(\beta-x_\beta)$.  But $$\gamma(\alpha+x_\alpha) = P(\alpha+x_\alpha)-\delta(\alpha+x_\alpha)=P(\alpha+x_\alpha)-\frac{1}{2}<P(\beta-x_\beta)-\frac{1}{2}=P(\beta-x_\beta)-\delta(\beta-x_\beta)=\gamma(\beta-x_\beta)$$ since the function $P(x)$ is increasing on $[-3/2,\infty)$ and $\alpha+x_\alpha< \beta-x_\beta$.
\end{proof}

Let $q\in \Q_{\geq 0}$ be fixed.  Since $\gamma$ is increasing, we conclude that there exists a non-exceptional stable bundle $V$ of slope $\mu\geq 0$ with $$q \leq \chi(V)/\rk(V) \leq \gamma(\mu)$$ if and only if $\gamma^{-1}(q)\leq \mu$.  Thus the set of slopes of non-exceptional stable bundles with $\mu(V)\geq 0$ and $\chi(V)/\rk(V)\geq q$ has a minimum if and only if $\gamma^{-1}(q)$ is rational.   That this is always the case follows from our investigation into the number theory of exceptional slopes, as we shall now see.  

\begin{theorem}
The function $\gamma:\Q_{\geq 0} \to \Q_{\geq 0}$ is a bijection.  
\end{theorem}

We warn the reader that the properties of Proposition \ref{gammaProp} alone are not sufficient to imply the theorem.  In fact, it is a priori possible that there exists some $q\in \Q_{\geq 0}$ such that $\gamma^{-1}(q)$ lies in the generalized Cantor set $C = \R \sm \bigcup_{\alpha} I_\alpha$.  See \cite{MO1} for a discussion of several counterexamples.
 
\begin{proof}
Let $q\in \Q_{\geq 0}$, and put $\xi = \gamma^{-1}(q)$.  If $\xi\in I_\alpha$ for some $\alpha\in \F E$, then since $\gamma$ is piecewise linear with rational coefficients on $I_\alpha$ we conclude that $\xi\in \Q$, and we are done.  If $\xi$ lies in no $I_\alpha$, then we have $$q=\gamma(\xi) = P(\xi) - \delta(\xi) = P(\xi) - \frac{1}{2} = \frac{1}{2}(1+3\xi+\xi^2),$$ and thus $$\xi=\frac{-3+\sqrt{5+8q}}{2}.$$ Since $q\geq 0$, this contradicts Corollary \ref{repeatingFracCor} unless $q=0$; in case $q=0$ we trivially have $\xi=0$, so we are done.
\end{proof}

\begin{remark}\label{xiRemark}
In fact, keeping notation from the proof of the theorem, it is easy to show that if $\alpha$ is the associated exceptional slope to $\gamma^{-1}(q)$ then $\frac{1}{2}(-3+\sqrt{5 +8q})\in I_\alpha$.  Once the associated exceptional slope to $\gamma^{-1}(q)$ is known, it is easy to determine $\gamma^{-1}(q)$ precisely via the formulas in the proof of Proposition \ref{gammaProp}.  Thus a fast method for determining $\gamma^{-1}(q)$ is to first find the interval $I_\alpha$ in which $\frac{1}{2}(-3+\sqrt{5 +8q})$ lies.
\end{remark}

\begin{corollary}\label{nonExceptionalMinCor}
For every $q\in \Q_{\geq 0}$, the set of nonnegative slopes of non-exceptional stable bundles $V$ satisfying $\chi(V)/\rk(V) \geq q$ has a minimum, namely $\gamma^{-1}(q)$.   Furthermore, any such $V$ of minimum possible slope has $\chi(V)/\rk(V) = q$.
\end{corollary}

We now turn to including the exceptional bundles into the discussion.  These bundles have an unusually large ratio $\chi_\alpha/r_\alpha$ for their slopes, so they require special attention.  We first collect some more necessary facts about the exceptional slopes in the following easy lemma.

\begin{lemma}\label{exceptionalIneqLemma}
Let $\alpha\in \F E$ be nonnegative, and let $n$ be a nonnegative integer.
\begin{enumerate} \item We have $$\gamma(\alpha)<\frac{\chi_\alpha}{r_\alpha} < \gamma(\alpha+x_\alpha).$$ \item If $\gamma(\alpha) <n$ then $\chi_\alpha/r_\alpha \leq n.$ 
\item It is never the case that $\gamma(\alpha) = n$ unless $\alpha$ is an integer; in this case $$n = \frac{(\alpha+2)(\alpha+1)}{2}-1.$$
\end{enumerate}
\end{lemma}
\begin{proof}
(1) Note that $$\frac{\chi_\alpha}{r_\alpha}-\gamma(\alpha)=\frac{1}{r_\alpha^2}.$$ The derivative of the function $\gamma(\mu)$ on the interval $[\alpha,\alpha+x_\alpha)$ is $\alpha+3$, so it suffices to show $$\frac{1}{r_\alpha^2} < x_\alpha(\alpha+3).$$  In fact, even $$\frac{1}{r_\alpha^2}< 3x_\alpha$$ is true, and is easily verified by basic algebra.  

(2)  Assume that $\chi_\alpha/r_\alpha>n$.  Since $\chi_\alpha$ is an integer, we must actually have $$\frac{\chi_\alpha}{r_\alpha} \geq n + \frac{1}{r_\alpha}.$$ But then $$\gamma(\alpha) = \frac{\chi_\alpha}{r_\alpha}-\frac{1}{r_\alpha^2} \geq n$$ since $r_\alpha\geq 1$.

(3)  We compute $$\gamma(\alpha) = \frac{-1+r_\alpha^2+3(\alpha r_\alpha)r_\alpha+\alpha^2 r_\alpha^2}{2r_\alpha^2}.$$ The numerator of this expression is congruent to $-1+\alpha^2 r_\alpha^2 \pmod{r_\alpha}$, so since $\alpha^2 r_\alpha^2 \equiv -1 \pmod{r_\alpha}$ by Corollary \ref{congruentCor} we see that $\gamma(\alpha)$ is not an integer, except possibly when $\alpha$ is an integer or half-integer.  However, when $r_\alpha = 2$ we compute that the numerator is congruent to $2$ mod $4$.  The expression for $\gamma(\alpha)$ when $\alpha$ is an integer is elementary.
\end{proof}

We now remove our ``non-exceptional" hypothesis from Corollary \ref{nonExceptionalMinCor}.

\begin{theorem}\label{minSlopeThm}
Let $q\in \Q_{\geq 0}$.  The set of nonnegative slopes of stable bundles $V$ with the property $\chi(V)/\rk(V) \geq q$ has a minimum $\mu$.  

In case $q=n$ is a positive integer, there exist stable $V$ of slope $\mu$ with $\chi(V)/\rk(V)=n$.  In fact, unless $n$ is of the form $(r+2)(r+1)/2-1$ for a positive integer $r$, no stable $V$ of slope $\mu$ has $\chi(V)/\rk(V) > n$.
\end{theorem}

\begin{proof}
We already know from Corollary \ref{nonExceptionalMinCor} that the set of nonnegative slopes of non-exceptional stable bundles with the property $\chi(V)/\rk(V)\geq q$ has a minimum $\lambda = \gamma^{-1}(q)$.  Let $\alpha\in \F E$ be the associated exceptional slope to $\lambda$.  If $\beta\in \F E$ has  $0\leq \beta< \alpha$, then by Lemma \ref{exceptionalIneqLemma} (1) $$\frac{\chi_\beta}{r_\beta}<\gamma(\beta+x_\beta)<\gamma(\lambda)=q,$$ so the exceptional bundle $E_\beta$ does not have $\chi(E_\beta)/\rk(E_\beta) \geq q$.  On the other hand, it is possible that $\chi_\alpha/r_\alpha \geq q$ and $\alpha < \lambda$.  It follows that the minimum nonnegative slope $\mu$ of stable bundles with $\chi(V)/\rk(V)\geq q$ is either $\alpha$ or $\lambda$; at any rate, the minimum exists.

Now suppose $q=n$ is a positive integer.  With the notation of the previous paragraph, if $\mu = \lambda$ then it follows from Corollary \ref{nonExceptionalMinCor} that every non-exceptional stable bundle of slope $\mu$ with $\chi(V)/\rk(V)\geq n$ actually has $\chi(V)/\rk(V) = n$.  If in fact we have $\mu=\lambda=\alpha$, then we must have $\gamma(\alpha) = n$, so $n$ is of the form $(r+2)(r+1)/2-1$ by Lemma \ref{exceptionalIneqLemma} (3).  Finally, in case $\mu = \alpha\neq \lambda$ we see $\alpha <\lambda$, so $\gamma(\alpha)< n$.  By Lemma \ref{exceptionalIneqLemma} (2), we conclude $\chi_\alpha/r_\alpha \leq n$, and thus $\chi_\alpha/ r_\alpha = n$.
\end{proof}

At last, we can make precise our notion of the associated exceptional bundle to the ideal sheaf $I_Z$ of $n$ general points.

\begin{definition}
For an integer $n$, let $\mu$ be the minimum nonnegative slope of a stable bundle $V$ with $\chi(V)/\rk(V) =n$.  Let $\alpha\in \F E$ be the associated exceptional slope to $\mu$.  The \emph{associated exceptional slope} to  $\P^{2[n]}$ is $\alpha$.
\end{definition}

\section{Resolution of the ideal sheaf of $n$ points}\label{resolutionSection}

In this section we exhibit a particularly nice resolution of the ideal sheaf $I_Z$ of $n$ general points in $\P^2$.  The strategy will be to first describe the resolution for some $Z\in \P^{2[n]}$, then argue that in fact the general $Z$ has a resolution of this form as well.

Let $n\geq 1$ be fixed for this section, and until further notice 
$$\emph{assume that $n$ is not of the form ${r+2\choose 2}-1$}.$$ We will handle this very easy case later; while it is possible to handle it uniformly with the other cases, incorporating it into the main dialog makes things much more confusing due to its exceptional nature with respect to Theorem \ref{minSlopeThm}. 

We begin by letting $\mu$ be the minimum nonnegative slope of a stable bundle with the property $\chi/r \geq n$.  Consider the exceptional slope associated to $\mu$.  We may write it in the form $\alpha.\beta$ for some $\alpha,\beta\in \F E$ with $$\alpha = \varepsilon\left(\frac{p\vphantom1}{2^q}\right) \qquad \beta = \varepsilon\left(\frac{p+1}{2^q}\right).$$ The triples $(E_{\beta-3},E_{\alpha},E_{\alpha.\beta})$ and $(E_{\alpha.\beta}, E_\beta, E_{\alpha+3})$ are then both triads.  

\begin{lemma}\label{rankLemma}
Suppose $\mu\neq \alpha.\beta$.  (Since $n \neq {r+2\choose 2}-1$, this is equivalent to assuming $\gamma(\mu)=n$).
\begin{enumerate}
\item If $\mu \in (\alpha.\beta-x_{\alpha.\beta},\alpha.\beta)$, there exists a stable bundle $V$ of slope $\mu$ with $\chi(V)/\rk(V) = n$ and $\rk(V) = (\alpha.\beta)r_{\alpha.\beta}$.

\item If $\mu\in (\alpha.\beta,\alpha.\beta+x_{\alpha.\beta})$, there exists a stable bundle $V$ of slope $\mu$ with $\chi(V)/\rk(V) = n$ and $\rk(V) = (\alpha.\beta+3)r_{\alpha.\beta}.$
\end{enumerate}
\end{lemma}
\begin{proof}
Suppose $\mu\in (\alpha.\beta-x_{\alpha.\beta},\alpha.\beta)$.  By Proposition \ref{gammaProp}, we compute $$\gamma(\mu) = (\alpha.\beta)(\mu+3)+1+\Delta_{\alpha.\beta}-P(\alpha.\beta)=n,$$ so $$\mu = \frac{n-1+P(\alpha.\beta)-\Delta_{\alpha.\beta}}{\alpha.\beta}-3=\frac{(n-1) r_{\alpha.\beta}+\chi_{\alpha.\beta}}{(\alpha.\beta)r_{\alpha.\beta}}-3,$$ and it follows that $\mu(\alpha.\beta)r_{\alpha.\beta}$ is an integer.  We conclude by Theorem \ref{deltaClassification} that the necessary $V$ exists.  The argument when $\mu\in (\alpha.\beta,\alpha.\beta+x_{\alpha.\beta})$ is analogous.
\end{proof}

\begin{definition}
In case $\mu\neq \alpha.\beta$, any bundle provided by Lemma \ref{rankLemma} is called an \emph{associated orthogonal bundle} to the general ideal sheaf $I_Z$ of $n$ points.  When $\gamma(\mu)<n$, we have $\chi_{\alpha.\beta}/r_{\alpha.\beta}=n$, and the exceptional bundle $E_{\alpha.\beta}$ is called the associated orthogonal bundle.
\end{definition}

We denote by $V$ an associated orthogonal bundle to $I_Z$.  The terminology comes from the numerical expectation that if $I_Z$ is general then $H^i(V \te I_Z) = 0$ for all $i$.  Suppose $V$ is not exceptional.  In this case, we have $$P(\mu)-\delta(\mu) = \gamma(\mu) = n = \frac{\chi(V)}{\rk(V)} = P(\mu)-\Delta(V) ,$$ so $\Delta(V) = \delta(\mu)$ and $V$ has height zero.  Thus $V$ admits a nice resolution by semi-exceptional bundles.

\begin{proposition}\label{orthoResProp}
If $V$ is an associated orthogonal bundle, one of the following three possibilities holds.
\begin{enumerate}
\item $\mu = \alpha.\beta$ and $V$ is exceptional.
\item $\mu< \alpha.\beta$, and $\rk(V) = (\alpha.\beta)r_{\alpha.\beta}$.  In this case, $V$ admits a resolution $$0 \to E_{\beta-3}^{m_1} \to E_{\alpha}^{m_2}\to V \to 0,$$ where $m_1 = r_\alpha(\mu-\alpha)(\alpha.\beta)$ and $m_2=r_\beta(\mu-\beta+3)(\alpha.\beta)$.  In particular, the numbers $r_\alpha(\mu-\alpha)(\alpha.\beta)$ and $r_\beta(\mu-\beta+3)(\alpha.\beta)$ are positive integers.  Furthermore, the inequalities $$r_{\alpha.\beta} x_{\alpha.\beta} < \frac{m_1}{m_2}\leq \frac{r_{\alpha.\beta}}{3 r_\beta r_{\alpha.\beta}-r_\alpha}$$ hold.
\item $\mu> \alpha.\beta$, and $\rk(V) = (\alpha.\beta+3)r_{\alpha.\beta}$.    In this case, $V$ admits a resolution $$0\to V \to E_\beta^{m_2} \to E_{\alpha+3}^{m_1} \to 0$$ with $m_2=r_\alpha(3+\alpha-\mu)(\alpha.\beta+3)$ and $m_1=r_{\beta}(\beta-\mu)(\alpha.\beta+3)$, and the inequalities $$r_{\alpha.\beta} x_{\alpha.\beta} < \frac{m_1}{m_2} \leq \frac{ r_{\alpha.\beta}}{3r_\alpha r_{\alpha.\beta}-r_\beta}$$ hold.
\end{enumerate}
\end{proposition}
\begin{proof}
Suppose we are in case (2), so that $\mu< \alpha.\beta$.  Since $V$ has height zero, associated exceptional slope $\alpha.\beta$,  and $(E_{\beta-3},E_\alpha,E_{\alpha.\beta})$ is a triad, there exists a resolution of the form $$0\to E_{\beta-3}^{m_1}\to E_{\alpha}^{m_2}\to V\to 0.$$  It follows that $$\alpha = \mu(E^{m_2}_\alpha) = \frac{c_1(E_{\beta-3}^{m_1})+c_1(V)}{ \rk(E_{\beta-3}^{m_1})+\rk(V)}=\frac{m_1(\beta-3)r_\beta+\mu (\alpha.\beta)r_{\alpha.\beta}}{m_1r_\beta+(\alpha.\beta)r_{\alpha.\beta}},$$ and by basic algebra $$m_1 = \frac{r_{\alpha.\beta}(\mu-\alpha)(\alpha.\beta)}{r_\beta(3+\alpha-\beta)}=r_\alpha(\mu-\alpha)(\alpha.\beta),$$ where we made use of the identity $r_{\alpha.\beta} = r_\alpha r_\beta (3+\alpha-\beta)$.  Comparing ranks, we have $$m_2r_\alpha=m_1r_\beta +(\alpha.\beta)r_{\alpha.\beta},$$ and using $r_{\alpha.\beta} = r_\alpha r_\beta (3+\alpha-\beta)$ again we conclude $m_2 = r_\beta(\mu-\beta+3)(\alpha.\beta)$.  The inequalities on $m_1/m_2$ follow immediately from Section \ref{triadSection}.

In case (3) we use the triad $(E_{\alpha.\beta},E_\beta,E_{\alpha+3})$ and perform an identical calculation.
\end{proof}

The discussion of Section \ref{triadSection} shows that the prior resolution of $V$ is canonical, and that the numbers $m_1$ and $m_2$ are suitable Euler characteristics.

\subsection{The resolution of $I_Z$ for $\mu < \alpha.\beta$} At this point we ``guess'' a resolution of an ideal sheaf $I_Z$.  First let us suppose $\mu< \alpha.\beta$ and $\rk(V) = (\alpha.\beta)r_{\alpha.\beta}$. Put $m_1 = r_\alpha(\mu-\alpha)(\alpha.\beta)$ and $m_2=r_\beta(\mu-\beta+3)(\alpha.\beta)$, so that we have a resolution $$0\to E_{\beta-3}^{m_1} \to E_{\alpha}^{m_2}\to V \to 0.$$  
From the inequality $$\frac{m_1}{m_2} > r_{\alpha.\beta}x_{\alpha.\beta}$$ we find $3r_{\alpha.\beta}m_1>m_2$ (proving this reduces to the inequality $3x_{\alpha.\beta}r_{\alpha.\beta}^2>1$ already considered in the proof of Lemma \ref{exceptionalIneqLemma} (1)).

Thus we may consider a coherent sheaf $W$ defined by the exact sequence $$0\to W \to E_{-\alpha-3}^{m_1} \fto{\phi} E_{-\beta}^{3r_{\alpha.\beta} m_1-m_2}\to 0$$ where the map $\phi$ between semi-exceptional bundles is general.  To ensure that $\phi$ is surjective, we restrict our attention to the case  where the ``expected rank" $m_1r_\alpha - (3r_{\alpha.\beta}m_1-m_2)r_\beta$ of $W$ is at least $2$; we will see later that this is not very restrictive at all, and will handle the other cases separately.  Despite the form of its resolution, we warn that $W$ is \emph{not} typically a height zero bundle.  By Proposition \ref{bertiniProp}, if we show $\sHom(E_{-\alpha-3},E_{-\beta})$ is globally generated, the hypothesis that the expected rank of $W$ is at least $2$ implies $\phi$ is surjective and $W$ is locally free of the expected rank.

\begin{lemma}
Let $\alpha,\beta$ be exceptional slopes of the form $$\alpha = \varepsilon\left(\frac{p\vphantom1}{2^q}\right) \qquad \beta = \varepsilon\left(\frac{p+1}{2^q}\right).$$ Then the sheaf $\sHom(E_{\beta-3},E_{\alpha})$ is globally generated.
\end{lemma}
\begin{proof}
We induct on $q$, starting with $q=-1$ as a base case.  Clearly if $\alpha = k$ and $\beta = k+2$ for some integer $k$ then $\sHom(E_{\beta-3},E_\alpha) \cong \OO_{\P^2}(1)$ is globally generated.

So suppose $q\geq 0$, and first assume $p$ is odd.  Putting $\eta = \varepsilon((p-1)/2^q)$, the induction hypothesis shows $\sHom(E_{\beta-3},E_\eta)$ is globally generated.  We have $\alpha = \eta.\beta$, so $(E_\eta,E_\alpha,E_\beta)$ is a triad.  The canonical map $$E_{\eta}\te \Hom(E_{\eta},E_{\alpha})\to E_{\alpha}\to 0$$ is surjective, so $$\sHom(E_{\beta-3},E_\eta)\te \Hom(E_\eta,E_\alpha)\to \sHom(E_{\beta-3},E_\alpha)\to 0$$ is also surjective.  Thus $\sHom(E_{\beta-3},E_\alpha)$ is a quotient of a globally generated bundle, and it is globally generated.

When $p$ is even, one reduces to the odd case by replacing $(\alpha,\beta)$ with $(-\beta,-\alpha)$ and then noting $\sHom(E_{-\alpha-3},E_{-\beta})\cong \sHom(E_{\beta-3},E_{\alpha})$.
\end{proof}

Next, put $m_3 = r_{\alpha.\beta}(P(\alpha.\beta)-\Delta_{\alpha.\beta}-n) = \chi_{\alpha.\beta}-r_{\alpha.\beta}n=\chi(E_{\alpha.\beta}\te I_Z)$ (which is clearly a positive integer) and consider a general map $$W \fto{\psi} E_{-(\alpha.\beta)}^{m_3}.$$ We will see in a moment by a purely numerical calculation that the expected rank of $W$ equals $\rk(E^{m_3}_{-(\alpha.\beta)})-1$, so our hypothesis on the expected rank of $W$ is equivalent to requiring $m_3 r_{\alpha.\beta}\geq 3$. In this case, $\sHom(W,E_{-(\alpha.\beta)})$ is globally generated.  Indeed, from the defining sequence of $W$ we have a surjection $$\sHom(E_{-\alpha-3}^{m_1},E_{-(\alpha.\beta)})\to \sHom(W,E_{-(\alpha.\beta)})\to 0,$$ and global generation of the former bundle follows from the lemma.  Thus the rank of $\psi$ is only less than $\rk W$ in codimension $2$, and it is an injection.  

\begin{lemma}\label{chernCalcLemma}
Assume $m_3r_{\alpha.\beta}\geq 3$, and consider the cokernel $Q$ in the exact sequence $$0\to W\to E^{m_3}_{-(\alpha.\beta)}\to Q\to 0,$$ where $W$ is given by an exact sequence$$0\to W \to E_{-\alpha-3}^{m_1} \to E_{-\beta}^{3r_{\alpha.\beta} m_1-m_2}\to 0.$$ We have $\rk (Q) = 1$, $c_1(Q)=0$, and $\ch_2(Q)=-n$.
\end{lemma}
\begin{proof}
The proof is an entirely numerical calculation, albeit a complicated one.    The basic strategy is to use the identities of Lemma \ref{numericalProps} repeatedly to remove instances of $r_\alpha,r_\beta,r_{\alpha.\beta}$, and $\alpha.\beta$ from expressions. Eventually we arrive at an expression only involving $\alpha$ and $\beta$, which turns out to not require any special properties of $\alpha,\beta$.     We first recollect 
 \begin{eqnarray*} 
 m_1 & = & r_\alpha(\mu-\alpha)(\alpha.\beta)\\
 m_2 & = & r_\beta(\mu-\beta+3)(\alpha.\beta)\\
 m_3 & = & r_{\alpha.\beta}(P(\alpha.\beta)-\Delta_{\alpha.\beta}-n)
 \end{eqnarray*}
 and also $$n = \gamma(\mu) = (\alpha.\beta)(\mu+3)+1+\Delta_{\alpha.\beta}-P(\alpha.\beta).$$
 
Let us show $\rk(E_{-(\alpha.\beta)}^{m_3})=1+\rk(W)$.  In fact, we will show this equality holds when $\rk(W)$ is interpreted as the expected rank of $W$, so that the condition $m_3r_{\alpha.\beta}\geq 3$ ensures the expected rank of $W$ is at least $2$.  Expanding, we have \begin{eqnarray*}\rk(E^{m_3}_{-(\alpha.\beta)})&=& m_3r_{\alpha.\beta}\\&=& r_{\alpha.\beta}^2(P(\alpha.\beta)-\Delta_{\alpha.\beta}-n)\\
&=& r_{\alpha.\beta}^2(2P(\alpha.\beta)-2\Delta_{\alpha.\beta}-(\mu+3)(\alpha.\beta)-1)
\end{eqnarray*} and \begin{eqnarray*}1+\rk(W) &=& 1+\rk(E^{m_1}_{-\alpha-3})-\rk(E_{-\beta}^{3r_{\alpha.\beta}m_1-m_2})\\
&=& 1+m_1r_\alpha-(3r_{\alpha.\beta}m_1-m_2)r_\beta\\
&=& 1+r_{\alpha}^2(\mu-\alpha)(\alpha.\beta)-(3r_{\alpha.\beta}r_\alpha(\mu-\alpha)(\alpha.\beta)-r_\beta(\mu-\beta+3)(\alpha.\beta))r_\beta\\
&=& 1+ (r_\alpha^2-3r_{\alpha.\beta} r_\alpha r_\beta+r_\beta^2)(\mu-\alpha)(\alpha.\beta)+r_\beta^2(3+\alpha-\beta)(\alpha.\beta)\\
&=& 1+(r_\alpha^2+r_\beta^2-3r_\alpha^2r_\beta^2(3+\alpha-\beta))(\mu-\alpha)(\alpha.\beta)+r_\beta^2(3+\alpha-\beta)(\alpha.\beta)\\
&=& 1+r_\alpha^2r_\beta^2\left(\frac{1}{r_\beta^2}+\frac{1}{r_\alpha^2}-3(3+\alpha-\beta)\right)(\mu-\alpha)(\alpha.\beta)+r_\beta^2(3+\alpha-\beta)(\alpha.\beta)\\
&=& 1+r_\alpha^2r_\beta^2(2-2\Delta_\beta-2\Delta_\alpha-3(3+\alpha-\beta))(\mu-\alpha)(\alpha.\beta)+r_{\beta}^2(3+\alpha-\beta)(\alpha.\beta)\\
&=& 1+r_\alpha^2r_\beta^2(2-2P(\alpha-\beta)-3(3+\alpha-\beta))(\mu-\alpha)(\alpha.\beta)+r_\beta^2(3+\alpha-\beta)(\alpha.\beta)\\
&=& 1-r_\alpha^2r_\beta^2(3+\alpha-\beta)^2(\mu-\alpha)(\alpha.\beta)+r_\beta^2(3+\alpha-\beta)(\alpha.\beta)\\
&=& 1-r_{\alpha.\beta}^2(\mu-\alpha)(\alpha.\beta)+r_\beta^2(3+\alpha-\beta)(\alpha.\beta)
\end{eqnarray*}
Notice that our final expressions for $\rk(E_{-(\alpha.\beta)}^{m_3})$ and $1+\rk(W)$ both have the same coefficient of $\mu$.  Thus we are reduced to showing $$r_{\alpha.\beta}^2(2P(\alpha.\beta)-2\Delta_{\alpha.\beta}-3(\alpha.\beta)-1)=1+r_{\alpha.\beta}^2 \alpha(\alpha.\beta)+r_\beta^2(3+\alpha-\beta)(\alpha.\beta),$$ an identity only involving exceptional slopes.  Equivalently, we must show  
$$\frac{r_{\alpha.\beta}^2(2P(\alpha.\beta)-2\Delta_{\alpha.\beta} -(3+\alpha)(\alpha.\beta)-1)}{r_\beta^2(3+\alpha-\beta)}
=\frac{1}{r_\beta^2(3+\alpha-\beta)}+\alpha.\beta$$ or, applying Lemma \ref{numericalProps}, $$r_\alpha^2 (3+\alpha-\beta)(2P(\alpha.\beta)-2\Delta_{\alpha.\beta}-(3+\alpha)(\alpha.\beta)-1)=\beta.$$ Since $$\frac{\beta}{r_\alpha^2(3+\alpha-\beta)} = -\frac{1}{r_\alpha^2}+\frac{3+\alpha}{r_\alpha^2(3+\alpha-\beta)}=2\Delta_\alpha-1+\frac{3+\alpha}{r_\alpha^2(3+\alpha-\beta)}=2\Delta_\alpha-1+(3+\alpha)(\alpha.\beta-\alpha)$$ this is the same as showing $$ 2P(\alpha.\beta)-2\Delta_{\alpha.\beta}-2\Delta_{\alpha}-(3+\alpha)(\alpha.\beta)=(3+\alpha)(\alpha.\beta-\alpha).$$ Now using $P(\alpha-\alpha.\beta) = \Delta_\alpha+\Delta_{\alpha.\beta}$, we reduce to verifying $$2P(\alpha.\beta)-2P(\alpha-\alpha.\beta)-(3+\alpha)(\alpha.\beta)=(3+\alpha)(\alpha.\beta-\alpha).$$ This final equality is true with any numbers $x,y$ in place of $\alpha$ and $\alpha.\beta$, so we conclude the required identity of ranks.

To compute $c_1(Q)$, we perform a similar calculation.  In case $\alpha.\beta = 0$, observe $c_1(Q) = 0$ is obvious, so we divide by $\alpha.\beta$ freely.  On the one hand, $$
\frac{c_1(E^{m_3}_{-(\alpha.\beta)})}{\alpha.\beta}= \rk(E_{-(\alpha.\beta)}^{m_3})$$
We also have 
\begin{eqnarray*} 
\frac{c_1(W)}{\alpha.\beta} &=& \frac{1}{\alpha.\beta}\left(m_1c_1(E_{-\alpha-3})-(3r_{\alpha.\beta}m_1-m_2)c_1(E_{-\beta})\right)\\
&=& \frac{1}{\alpha.\beta}(-m_1r_\alpha (\alpha+3)+(3r_{\alpha.\beta} m_1 - m_2)r_\beta \beta)\\
&=& - r_\alpha^2(\mu-\alpha)(\alpha+3)+(3r_{\alpha.\beta}r_\alpha(\mu-\alpha)-r_\beta(\mu-\beta+3))r_\beta \beta\\
&=& -r_\alpha^2(\mu-\alpha)(\alpha+3)+3r_{\alpha.\beta}r_\alpha r_\beta (\mu-\alpha)\beta-r_\beta^2(\mu-\alpha)\beta-r_\beta^2(3+\alpha-\beta)\beta\\
&=& (\mu-\alpha)(-r_\alpha^2(\alpha+3)+3r_\alpha^2r_\beta^2 (3+\alpha-\beta)\beta-r_\beta^2\beta)-r_\beta^2(3+\alpha-\beta)\beta\\
&=& (\mu-\alpha)(-r_\alpha^2(3+\alpha-\beta)+3r_\alpha^2r_\beta^2(3+\alpha-\beta)\beta-(r_\alpha^2+r_\beta^2)\beta)-r_\beta^2(3+\alpha-\beta)\beta
\end{eqnarray*}
Let us verify the identity  $$-r_{\alpha}^2(3+\alpha-\beta)+3r_{\alpha}^2r_\beta^2(3+\alpha-\beta)\beta-(r_\alpha^2+r_\beta^2)\beta=r_{\alpha.\beta}^2(\alpha.\beta).$$ Dividing both sides by $r_{\alpha.\beta}^2$ and applying $P(\alpha-\beta)   = \Delta_\alpha+\Delta_\beta$ shows that it is equivalent to $$-\beta + \frac{3\beta}{3+\alpha-\beta}+\frac{\beta(2P(\alpha-\beta)-2)}{(3+\alpha-\beta)^2}=0,$$ which is valid for any real numbers $\alpha,\beta$ with $3+\alpha-\beta\neq 0$.  Thus
\begin{eqnarray*} \frac{c_1(W)}{\alpha.\beta}&=& r_{\alpha.\beta}^2(\mu-\alpha)(\alpha.\beta)-r_\beta^2(3+\alpha-\beta)\beta\\&=& r_{\alpha.\beta}^2(\mu-\alpha)(\alpha.\beta)-r_\beta^2(3+\alpha-\beta)(\alpha.\beta)+r_\beta^2(3+\alpha-\beta)(\alpha.\beta-\beta)\\
&=& r_{\alpha.\beta}^2(\mu-\alpha)(\alpha.\beta)-r_\beta^2(3+\alpha-\beta)(\alpha.\beta)-1\\
&=& -(1+\rk(W)),
\end{eqnarray*} comparing with our final expression for $1+\rk(W)$ in the previous calculation.  Since we already know $\rk(E^{m_3}_{-(\alpha.\beta)})= 1+\rk(W)$, we conclude $c_1(Q) = 0$.

It is possible at this point to prove $\ch_2(Q) = -n$ by the same methods.  However, we will later show that there is a bundle $V$ with $\chi(V) = \rk(V)n$ such that $V\te Q$ is acyclic, and this implies $\ch_2(Q)=-n$.  That result will not make use of any of the further discussion in the rest of this section, so we may safely use this fact at will. 
\end{proof}

\begin{corollary}
Assume $m_3r_{\alpha.\beta}\geq 3$.  The sheaf $Q$ in the sequence $$0\to W \fto{\psi} E_{-(\alpha.\beta)}^{m_3}\to Q\to 0$$ is the ideal sheaf $I_Z$ of a zero-dimensional scheme $Z\subset \P^2$ of degree $n$.
\end{corollary}
\begin{proof}
If we can show $Q$ torsion free, then there is an embedding $Q\to Q^{**}$ which is an isomorphism outside of codimension $2$.  Since $\rk(Q)=1$ and $c_1(Q)=0$, we have $Q^{**} \cong \OO_{\P^2}$, and it follows that $Q \cong I_Z$ for some zero-dimensional subscheme $Z\subset \P^2$.  Its degree must be $n$ since $\ch_2(Q) = -n$.

To verify that $Q$ is torsion free, first observe that by construction the rank of $\psi$ only drops in codimension $2$.  Thus any torsion occurs in codimension $2$.  Let $T\subset Q$ be the torsion subsheaf.  If $T\neq 0$ then $h^0(T) >0$ (since $T$ has zero-dimensional support) and therefore $h^0(Q) > 0$.

However, we claim $h^0(Q)=0$.  For this it suffices to verify $h^0(E_{-(\alpha.\beta)})=0$ and $h^1(W)=0$.  For the first, we simply note that $-(\alpha.\beta)<0$ and $E_{-(\alpha.\beta)}$ is stable.  To see $h^1(W)=0$ it is enough to check $h^0(E_{-\beta})=0$ and $h^1(E_{-\alpha-3})=0$.  The first equality follows from stability.  To see the second, we have $H^1(E_{-\alpha-3}) = H^1(E_{\alpha}) = \Ext^1(\OO_{\P^2},E_{\alpha})$, which is zero by Theorem \ref{excepOrthogonalThm}.
\end{proof}

\subsection{The resolution of $I_Z$ for $\mu > \alpha.\beta$}\label{biggerHalf}
At this point let us indicate how the previous arguments carry over to the case where $\mu> \alpha.\beta$ and $\rk(V) = (\alpha.\beta+3)r_{\alpha.\beta}$.  Begin from the resolution $$0\to V\to E_{\beta}^{m_2}\to E_{\alpha+3}^{m_1}\to 0$$ with $m_2 = r_\alpha(3+\alpha-\mu)(\alpha.\beta+3)$ and $m_1=r_\beta(\beta-\mu)(\alpha.\beta+3)$.  This time we consider a general sheaf $$0\to E_{-\alpha-3}^{3r_{\alpha.\beta}m_1-m_2}\fto{\phi} E_{-\beta}^{m_1}\to W\to 0,$$ and assume the expected rank of $W$ is at least $2$ so that $\phi$ is injective and $W$ is locally free.  Let 
$m_3=-r_{\alpha.\beta}(P(\alpha.\beta)-\Delta_{\alpha.\beta}-n) = r_{\alpha.\beta}n-\chi_{\alpha.\beta} = -\chi(E_{\alpha.\beta}\te I_Z)> 0$, and look at a general map $$E_{-(\alpha.\beta)-3}^{m_3}\fto{\psi} W,$$ easily verifying the bundle of such maps is globally generated.  Verify by a numerical calculation that $\rk(W) = 1+ \rk(E^{m_3}_{-(\alpha.\beta)-3})$, so that $\psi$ is injective and the cokernel $Q$ has rank $1$; furthermore we conclude that the expected rank of $W$ is always at least $2$, so there are no exceptional cases to consider here.  Compute $c_1(Q) = 0$, and conclude as before that $Q$ is an ideal sheaf of a zero-dimensional subscheme $Z\subset \P^2$ of degree $n$.  

\subsection{The resolution of $I_Z$ for $\mu = \alpha.\beta$, with $V$ exceptional}  In case $V$ is exceptional, things are substantially easier.  We have $\gamma(\mu)<n$, so let $\lambda>\mu$ be the rational number with $\gamma(\lambda)=n$ (it has associated exceptional slope $\alpha.\beta$ by Lemma \ref{exceptionalIneqLemma} (1)), and let $U$ be a stable bundle of slope $\lambda$ and rank $(\alpha.\beta+3)r_{\alpha.\beta}$ with $\chi(U)/\rk(U) = n$.  Apply the discussion of Section \ref{biggerHalf} to $U$ instead of $V$ (and $\lambda$ instead of $\mu$); we see that $m_3=0$, and it follows from the numerical calculations that $W$ itself is already an ideal sheaf $I_Z$.  We can then further calculate $m_1 = \chi(E_\beta \te I_Z)$ and $3r_{\alpha.\beta} m_1-m_2 = -\chi(E_{\alpha}\te I_Z) = -\chi(I_Z,E_{-\alpha-3})$ using the techniques of the proof of Lemma \ref{chernCalcLemma} (where $m_1$ and $m_2$ are defined in terms of $\lambda$ instead of $\mu$).  Consider the resolution $$0\to E_{-\alpha-3}^{3r_{\alpha.\beta}m_1-m_2}\to E_{-\beta}^{m_1} \to I_Z\to 0.$$ Applying $\sHom(-,E_{-\alpha-3})$ to this sequence and taking cohomology shows that in fact $$-\chi(I_Z,E_{-\alpha-3}) = \dim \Ext^1(I_Z,E_{-\alpha-3})$$ since $E_\beta \te E_{-\alpha-3}$ is acyclic and $E_{-\alpha-3}$ is simple.  Similarly, applying $\sHom(E_{-\beta},-)$ and taking cohomology gives $$\chi(E_{-\beta},I_Z) = \dim\Hom(E_{-\beta},I_Z).$$ Thus we have a resolution $$0\to E_{-\alpha-3}\te \Ext^1(I_Z,E_{-\alpha-3})^*\to E_{-\beta} \te \Hom(E_{-\beta},I_Z)\to I_Z\to 0$$ in case $\mu = \alpha.\beta$.

\subsection{The remaining cases}  The only cases we have not yet covered are the cases $m_3r_{\alpha.\beta}\leq 2$ in case $\mu < \alpha.\beta$. We will discuss the case where $r_{\alpha.\beta}=2$ and $m_3=1$ in detail; the other cases can be handled in a similar but easier manner.

If $r_{\alpha.\beta}=2$ and $m_3=1$, we must have $(\alpha,\alpha.\beta,\beta)=(k,k+\frac{1}{2},k+1)$ for some positive integer $k$.  We have $$m_3 = r_{\alpha.\beta}(P(\alpha.\beta)-\Delta_{\alpha.\beta}-n)=1,$$ so $$n=\frac{1}{2}(k^2+4k+2).$$ In fact, $k=2l$ must be even for $n$ to be an integer.  Write $$n = \frac{(2l+1)(2l+2)}{2}+l.$$ Gaeta's theorem asserts that the ideal sheaf $I_Z$ of $n$ general points has a resolution $$0\to \OO_{\P^2}(-2l-2)\oplus \OO_{\P^2}(-2l-3)^{l}\fto{M} \OO_{\P^2}(-2l-1)^{l+2}\to I_Z\to 0,$$ where $M$ is a general matrix of forms of the appropriate degrees.  After applying automorphisms of the bundles, the matrix $M$ can be brought into the form 
$$\begin{pmatrix}
x & q_{11}& \cdots & q_{1l}\\
y & q_{21}& \cdots & q_{2l}\\
z & q_{31}& \cdots & q_{3l}\\
0 & q_{41} & \cdots & q_{41}\\
\vdots & \vdots & \ddots & \vdots\\
0 & q_{l+2} & \cdots & q_{l+2,l}
\end{pmatrix}$$
where the $q_{ij}$ are quadrics.  In other words, the resolution can be rewritten as $$0\to \OO_{\P^2}(-2l-3)^l\to \OO_{\P^2}(-2l-1)^{l-1}\oplus T_{\P^2}(-2l-2)\to I_Z\to 0.$$  In particular, there is a map $T_{\P^2}(-2l-2)\to I_Z$ (which is neither surjective nor injective).  Consider the map of complexes 
$$\xymatrix{
0\ar[r]&T_{\P^2}(-2l-2) \ar[r] & I_Z\ar[r] & 0\\
0\ar[r]&\OO_{\P^2}(-2l-3)^l \ar[u] \ar[r]& \OO_{\P^2}(-2l-1)^{l-1} \ar[u] \ar[r] & 0
}$$
where the maps all come from the resolution and $I_Z$ is placed in the degree $0$ position.  One can check this diagram commutes and is a quasi-isomorphism.  Denoting the second complex by $W^\bullet[1]$, this implies that $W^\bullet[1]$ is isomorphic to the mapping cone of the morphism $T_{\P^2}(-2l-2)\to I_Z$ in the derived category $D^b(\Coh(\P^2))$.  Thus there is a distinguished triangle $$W^{\bullet}\to T_{\P^2}(-2l-2)\to I_Z \to.$$ Note that we have proved this result for every general $Z$, and the other outstanding cases can also be handled for general $Z$ in this fashion.  

By working in the heart of a suitable $t$-structure on $D^b(\Coh(\P^2))$, it is possible recover exactness.    This idea will play a prominent role in Section \ref{bridgelandSection}.

We note that the case where $n = {r+2\choose 2}-1$ can also be handled by this method.  Carry out the procedure for $\mu<\alpha.\beta = r$ using a \emph{non-exceptional} stable bundle $V$ of rank $r$ having $\chi/r = n$.  It is necessary to interpret $W$ as a complex, but we obtain a distinguished triangle $$ W^{\bullet}\to \OO_{\P^2}(-r)\to I_Z\to.$$

Let us recap what has been proved to this point.

\begin{proposition}\label{firstResProp}
In case $\mu<\alpha.\beta$ or $n={\mu+2\choose 2}-1$ there is a distinguished triangle $$W^{\bullet}\to E_{-(\alpha.\beta)}^{m_3} \to I_Z\to$$ for some $Z\in \P^{2[n]}$, where $W^{\bullet}$ is a complex $$E_{-\alpha-3}^{m_1}\to E_{-\beta}^{3r_{\alpha.\beta}m_1-m_2}$$ concentrated in degrees $0$ and $1$.  So long as $\chi(E_{\alpha.\beta}\te I_Z)r_{\alpha.\beta}\geq 3$, $W^{\bullet}$ is a vector bundle (sitting in degree 0) and the distinguished triangle becomes an exact sequence.  

When $\mu>\alpha.\beta$ and $n \neq {\mu+2\choose 2}-1$, there is an exact sequence $$0\to E_{-(\alpha.\beta)-3}^{m_3}\to W\to I_Z\to 0$$ for some $I_Z$, where $W$ fits into an exact sequence $$0\to E_{-\alpha-3}^{3r_{\alpha.\beta}m_1-m_2}\to E_{-\beta}^{m_1}\to W\to 0.$$ 

In case $\mu = \alpha.\beta$, some $I_Z$ admits a resolution $$0\to E_{-\alpha-3}\te \Ext^1(I_Z,E_{-\alpha-3})^*\to E_{-\beta} \te \Hom(E_{-\beta},I_Z)\to I_Z\to 0.$$
\end{proposition}

\subsection{Interpolation for exceptional bundles}  We now know enough about resolutions of ideal sheaves to show that the general $I_Z$ imposes the ``expected'' number of conditions on sections of a large family of exceptional bundles.  This result will allow us to clarify the nature of our resolution of $I_Z$.

\begin{theorem}\label{exceptionalInterpThm}
Let $I_Z$ be a general ideal sheaf of $n$ points, and let $\alpha.\beta$ be the exceptional slope associated to the rational number $\mu$ with $\gamma(\mu)=n$.  Let $\eta$ be an exceptional slope which satisfies 
 \begin{enumerate}
 \item $\eta \leq \alpha$,
 \item $\eta = \alpha.\beta$, or
 \item $\eta  \geq \beta$.
 \end{enumerate}  
 If $H^0(E_{\eta}\te I_Z)\neq 0$, then $H^1(E_{\eta}\te I_Z) =0$.
\end{theorem}

That is, vanishing at a general collection of $n$ points imposes the expected number of conditions on sections of $E_\eta$.  We suspect the theorem is true for all $\eta$, but the cases where $\alpha < \eta <\beta$ and $\eta\neq \alpha.\beta$ are typically more difficult.

\begin{proof}
This is an open property of $Z$, so it suffices to show the result holds for a specific $Z$.   Suppose we are in the case $\mu < \alpha.\beta$ and $\chi(E_{\alpha.\beta}\te I_Z)r_{\alpha.\beta}\geq 3$, and consider the exact sequence $$0\to W \to E^{m_3}_{-(\alpha.\beta)}\to I_Z\to 0.$$ When $\eta \leq \alpha$, we claim $H^0(E_{\eta}\te I_Z)=0$.  We have $\Hom(E_{-\eta},E_{-(\alpha.\beta)})=0$ by stability since $-\eta > -(\alpha.\beta)$, so it is enough to show $\Ext^1(E_{-\eta},W)=0$.  From the exact sequence $$0\to W \to E_{-\alpha-3}^{m_1}\to E_{-\beta}^{3r_{\alpha.\beta}m_1-m_2}\to 0$$ we see that we must verify $\Hom(E_{-\eta},E_{-\beta})=0$ and $\Ext^1(E_{-\eta},E_{-\alpha-3})=0$.  The first group is zero since $-\eta >-\beta$, while the second is isomorphic to $\Ext^1(E_{-\alpha},E_{-\eta})$, which is zero by Theorem \ref{excepOrthogonalThm} because $-\alpha\leq -\eta.$  Thus $H^0(E_\eta\te I_Z)=0$.  The case where $\eta \geq \beta$ is similar.  In case $\eta = \alpha.\beta$, note that $E_{\alpha.\beta}\te W$ is acyclic since $E_{\alpha.\beta}\te E_{-\alpha-3}$ and $E_{\alpha.\beta}\te E_{-\beta}$ are both acyclic.  Thus $$H^1(E_{\alpha.\beta}\te I_Z)\cong \Ext^1(E_{-(\alpha.\beta)},E_{-(\alpha.\beta)})=0$$ by rigidity.

The other possibilities for $\mu$ are handled in the same way.
\end{proof}

In particular, we see that for the general ideal sheaf $I_Z$ we have $$\dim \Hom(E_{-(\alpha.\beta)},I_Z) = \chi(E_{\alpha.\beta}\te I_Z)=m_3$$ in case $\mu<\alpha.\beta$ and $$m_3 = -\chi(E_{\alpha.\beta}\te I_Z)=\dim \Ext^1(E_{-(\alpha.\beta)},I_Z)=\dim \Ext^1(I_Z,E_{-(\alpha.\beta)-3})$$ in case $\mu> \alpha.\beta$.

We now conclude the section with our final result on the resolution of $I_Z$.

\begin{theorem}\label{resTheorem}
Let $Z$ be a general collection of $n$ points.  
\begin{enumerate}
\item In case $\mu<\alpha.\beta$ and $\chi(E_{\alpha.\beta}\te I_Z) r_{\alpha.\beta}\geq 3$, we have a resolution $$0\to W\to E_{-(\alpha.\beta)} \te \Hom(E_{-(\alpha.\beta)},I_Z)\to I_Z\to 0 $$ where the map $E_{-(\alpha.\beta)}\te \Hom(E_{-(\alpha.\beta)},I_Z)\to I_Z$ is the canonical one.  The isomorphism class of $W$ depends only on $Z$, and $W$ is a stable bundle with resolution $$0\to W\to E_{-\alpha-3}^{m_1}\to E_{-\beta}^{3r_{\alpha.\beta}m_1-m_2} \to 0,$$ where $m_1 = r_\alpha(\mu-\alpha)(\alpha.\beta)$ and $m_2 = r_\beta(\mu-\beta+3)(\alpha.\beta)$.

\item In case $\mu> \alpha.\beta$, we have a resolution $$0\to E_{-(\alpha.\beta)-3}\te \Ext^1(I_Z,E_{-(\alpha.\beta)-3})^*\to W\to I_Z\to 0.$$ The isomorphism class of $W$ depends only on $Z$, and $W$ is a stable bundle with resolution $$0\to E_{-\alpha-3}^{3 r_{\alpha.\beta m_1-m_2}}\to E_{-\beta}^{m_1}\to W\to 0,$$ where $m_1=r_\beta(\beta-\mu)(\alpha.\beta+3)$ and $m_2=r_{\alpha}(3+\alpha-\mu)(\alpha.\beta+3)$.

\item In case $\mu = \alpha.\beta$, we have a resolution $$0\to E_{-\alpha-3}\te \Ext^1(I_Z,E_{-\alpha-3})^*\to E_{-\beta} \te \Hom(E_{-\beta},I_Z)\to I_Z\to 0.$$
\end{enumerate}
\end{theorem}

Recall that we have already discussed how to prove the natural analog of this theorem in the cases where $\mu<\alpha.\beta$ and $\chi(E_{\alpha.\beta}\te I_Z)r_{\alpha.\beta} \leq 2$ by working in the derived category.

\begin{proof}
Let us focus on cases (1) and (2); the third case is easier.  The key fact is that in either case a general bundle $W$ with resolution of the prescribed form is stable, and, conversely, a general stable bundle $W'\in M(\ch(W))$ admits a resolution of the same form as $W$.  This will follow from Brambilla \cite[Proposition 4.4 and Theorem 8.2]{Brambilla2} if we check $\Ext^1(E_{-\alpha-3},E_{-\beta}(-1))=0$ (the necessary inequality to apply their theorem follows immediately from our inequalities on $m_1/m_2$).  This vanishing guarantees that the bundle $W$ is \emph{prioritary}, i.e. that $\Ext^2(W,W(-1))=0$.

To prove the required vanishing $H^1(E_{\alpha+3}\te E_{-\beta-1})=0$, first observe that it is obvious in case $\alpha$ and $\beta$ are both integers.  Thus we may assume $\beta-\alpha < 1$.  Write $F = E_{\alpha+3}\te E_{-\beta-3}$ and observe that $F$ is acyclic and $\mu(F) = \alpha-\beta > -1$; we must show $H^1(F(2))=0$.  So let $C \subset \P^2$ be a plane conic, and consider the exact sequences $$\begin{array}{ccccccccc} 0 &\to & F &\to & F(2) & \to & F(2)|_C & \to & 0 \\ 0 & \to & F(-2) & \to &  F &\to & F|_C & \to & 0.\end{array}$$ Since $F$ is acyclic, $H^1(F(2))\cong H^1(F(2)|_C)$.  We know $H^1(F|_C)=0$ since $F$ is acyclic and $H^2(F(-2))=0$ (as $F(-2)$ is stable with slope greater than $-3$). But $H^1(F|_C)$ surjects onto $H^1(F(2)|_C)$, so we conclude $H^1(F(2)|_C)=0$.

Suppose we are in the case $\mu<\alpha.\beta$.  By using Proposition \ref{firstResProp}  we see that there is some $Z\in \P^{2[n]}$ such that there is a resolution $$0\to W\to E_{-(\alpha.\beta)}^{m_3}\to I_Z\to 0$$ with $W$ stable (since it is general by construction) and having the specified resolution.  By Theorem \ref{exceptionalInterpThm}, $m_3 = \dim \Hom(E_{-(\alpha.\beta)},I_Z)$, so after performing an appropriate identification of $\C^{m_3}$ with $\Hom(E_{-(\alpha.\beta)},I_Z)$ we see that either the map $E_{-(\alpha.\beta)}\te \C^{m_3}\to I_Z$ is the canonical one or there is some factor $E_{-(\alpha.\beta)}$ which maps to zero, and hence is a summand of $W$.  But $W$ is stable, so this is impossible, and the map is the canonical one.  

For a general $Z$ with $\dim \Hom(E_{-(\alpha.\beta)},I_Z) = \chi(E_{-(\alpha.\beta)},I_Z)$, we consider the canonical map $$E_{-(\alpha.\beta)} \te \Hom(E_{-(\alpha.\beta)},I_Z)\to I_Z.$$ The property that this map is surjective is open in $Z$, the property that the kernel is stable is open in $Z$, and the property that the kernel has the expected resolution is open in $Z$.  Thus defining $W$ to be the kernel, we obtain the desired resolution.

The case where $\mu> \alpha.\beta$ is similar.  Sheaves $W$ fitting into the sequence $$0\to E_{-(\alpha.\beta)-3}\te \C^{m_3} \to W\to I_Z\to 0$$ are classified by elements of $\Ext^1(I_Z,E_{-(\alpha.\beta)-3})^{m_3}$.  If the $m_3$ components of such an element do not form a basis for $\Ext^1(I_Z,E_{-(\alpha.\beta)-3})$, then $W$ will have $E_{-(\alpha.\beta)-3}$ as a direct summand and will not be stable.  When the components do form a basis, the isomorphism class of $W$ is independent of the choice of elements, as different choices merely amount to different identifications $\C^{m_3} \cong \Ext^1(I_Z,E_{-(\alpha.\beta)-3})^*$.
\end{proof}

\section{Orthogonality of Kronecker modules}\label{SteinerSection}

Let $N\geq 3$ be fixed for this section.  A \emph{general Steiner bundle} $E$ on $\P^{N-1}=\P V$ is a vector bundle admitting a resolution of the form $$0\to \OO_{\P^{N-1}}^b\fto M \OO_{\P^{N-1}}^a(1)\to E\to 0,$$ where the $a\times b$ matrix $M$ of linear forms is general.  Consider the following fundamental problem.  Given a second general Steiner bundle $$0\to \OO_{\P^{N-1}}^{b'}\to \OO_{\P^{N-1}}^{a'}(1)\to F\to 0,$$ compute the dimension of the space $\Hom(F,E)$.

Since $\Ext^i(\OO_{\P^{N-1}}(1),\OO_{\P^{N-1}})=0$ for $0\leq i\leq N-1$, it is easy to see that any homomorphism $F\to E$ lifts to a commutative diagram 
$$\xymatrix{
0\ar[r]&\OO^b_{\P^{N-1}}  \ar[r] & \OO^a_{\P^{N-1}}(1) \ar[r]&  E \ar[r] & 0\\
0\ar[r]&\OO^{b'}_{\P^{N-1}}  \ar[r]\ar[u] & \OO^{a'}_{\P^{N-1}}(1) \ar[r]\ar[u]&  F \ar[r]\ar[u] & 0,\\
}$$
and in particular determines a diagram
$$\xymatrix{
0\ar[r]&\OO^b_{\P^{N-1}}  \ar[r] & \OO^a_{\P^{N-1}}(1) \\
0\ar[r]&\OO^{b'}_{\P^{N-1}}  \ar[r]\ar[u] & \OO^{a'}_{\P^{N-1}}(1) \ar[u] \\
}$$
On the other hand, any diagram of the latter form induces a homomorphism $F\to E$, and these constructions are inverse to one another.  

\subsection{Kronecker modules} The matrix $M$ defining the Steiner bundle $E$ can be thought of as a linear map $e:B\te V\to A$, where $B,A$  are $b$- and $a$-dimensional vector spaces, respectively.  The preceding discussion shows that the space $\Hom(F,E)$ is naturally isomorphic the space of commutative diagrams of the form 
$$\xymatrix{
 B \te V \ar[r] & A  \\
 B' \te V  \ar[u]^{\beta\te \id} \ar[r] & A' \ar[u]^{\alpha}\\
}$$

Let $Q$ be the quiver with two vertices and $N$ arrows from the first vertex to the second, i.e. the $N$-arrowed Kronecker quiver.  A representation of $Q$ (or a \emph{Kronecker $V$-module}) assigns to each vertex a vector space and to each arrow a linear map from the first vector space to the second. A representation $e$ of $Q$ is therefore the same thing as a linear map $e : B\te V\to A$, where $B,A$ are vector spaces.  If $f:B'\te V\to A'$ is a second representation, then the homomorphisms $f\to e$ are precisely the commutative diagrams as above.  In particular, if $E,F$ are Steiner bundles with corresponding Kronecker $V$-modules $e,f$, then $\Hom_Q(f,e) \cong \Hom(F,E)$.

The \emph{dimension vector} of $e:B\te V\to A$ is the element $\udim e = (\dim B,\dim A)$ of $\N^2$.  The Euler characteristic of a pair $f,e$ of representations is defined by $$\chi(f,e) = \dim \Hom_Q(f,e) - \dim \Ext^1_Q(f,e);$$ all higher $\Ext$ terms vanish.  The Euler characteristic can be computed numerically in terms of dimension vectors; precisely, if $\udim e = (b,a)$ and $\udim f = (b',a')$ then $$ \chi(f,e) = b'b+a'a-Nb'a.$$  

Fix a dimension vector $(b,a)$ and vector spaces $B,A$ of dimensions $b$ and $a$. There is a natural action of $\SL(B)\times \SL(A)$ on the space $\P \Hom(B\te V,A)$.  Denote by $Kr(V,B,A) = Kr(N,b,a)$ the semi-stable objects in the GIT quotient of this action.  If $e:B\te V\to A$ is a Kronecker module, we will also denote by $Kr(\udim e) = Kr(V,B,A)$ the space corresponding to the dimension invariants of $e$.  For a nonzero module $e$, we define the \emph{slope} $\mu(e)\in[ 0,\infty ]$ to be the number $b/a$, interpreted as $\infty$ if $a=0$ and $b\neq 0$.  It is observed in \cite{Drezet} that the general Kronecker $V$-module with slope $\mu$ will be GIT-stable whenever $$\mu \in (\psi_N^{-1},\psi_N), \qquad \textrm{where} \qquad \psi_N = \frac{N+\sqrt{N^2-4}}2.$$

By work of Schofield and van den Bergh \cite{Schofield,SvdB}, stability of quiver representations can be detected by the existence of orthogonal representations.  We state their result in the special case of the Kronecker quiver.

\begin{theorem}[{\cite[Corollary 1.1]{SvdB}}]
A Kronecker $V$-module $e$ is GIT-semistable if and only if there is a nontrivial Kronecker $V$-module $f$ with $\Hom(f,e) = \Ext^1 (f,e) = 0$.
\end{theorem}

Several other authors have discussed similar results, such as Derksen-Weyman and \'Alvarez-C\'onsul-King \cite{DerksenWeyman,ConsulKing}. The following restatement of the theorem is immediate by the computation of the Euler form.

\begin{corollary}\label{kroneckerOrthogonal}
Consider Kronecker modules \begin{eqnarray*}
\C^b\te V & \fto{e} & \C^a\\
\C^{ka} \te V& \fto{f} & \C^{k(Na-b)}
\end{eqnarray*}where $e$ is semistable and $f$\ is general.  If $k$ is sufficiently large, then $\Hom(f,e)=0$.

In particular, the conclusion holds if $e$ is general and $\mu(e) \in (\psi_N^{-1},\psi_N)$.
\end{corollary}

\begin{remark}\label{KroneckerDivisor} Keep the notation from the corollary.  In \cite{Drezet} it is shown that $Kr(\udim e)$ has Picard group $\Z$.   The corollary implies that for a general $f$ the locus $$D_f = \{e':\Hom(f,e')\neq 0\} \subset Kr(\udim e)$$ forms a divisor, which must be a multiple of the generator of the Picard group.  Furthermore, for any $e'\in Kr(\udim e)$, the general divisor $D_f$ does not contain $e'$.
\end{remark}

\subsection{Orthogonality of quotients of semi-exceptional bundles}\label{ss-orth}
As a simple application of the orthogonality result for Kronecker modules, consider a triad $(E,G,F)$ of exceptional bundles on $\P^2$, and put $N = \dim \Hom(E,G)=\rk F$.  Let $V,W$ be general quotients of the form
$$\begin{array}{ccccccccc}0 & \to & E^{b}  & \to & G^a & \to & W & \to & 0\\
0 &\to  & E^{ka} &\to & G^{k(Na-b)} &\to & V &\to  & 0\end{array}$$ with $k$ sufficiently large.  Since $\sHom(G,E)$ is acyclic, homomorphisms $V\to W$ correspond to diagrams
$$\xymatrix{ 0 \ar[r] & E^b\ar[r] & G^a \\ 0 \ar[r] & E^{ka} \ar[r]\ar[u] & G^{k(Na-b)}. \ar[u]}$$  Alternately, $W$ corresponds to a general Kronecker $\Hom(E,G)^*$-module $e:\C^b\te \Hom(E,G)^*\to \C^a$, and $V$ corresponds to a general $f: \C^{ka}\te \Hom(E,G)^*\to \C^{k(Na-b)}.$  Since $E$ and $G$ are simple, $\Hom(V,W)\cong \Hom_Q(f,e)$.

\begin{corollary}\label{quotientOrthogonal}
If $b/a\in (\psi_{N}^{-1},\psi_N)$ and $k$ is sufficiently large, then $\sHom(V,W)$ has no cohomology.
\end{corollary}
\begin{proof}
The inequalities on $b/a$ ensure that $V$ and $W$ are stable (as in the proof of Theorem \ref{resTheorem}).  Then $\sHom(V,W)$ is stable of slope $$\mu(W)-\mu(V) = \frac{(Nab - a^2 - b^2)\rk(E)\rk(G)(\mu(G)-\mu(E))}{\rk(W)\rk(V)},$$ which is nonnegative by the hypothesis on $b/a$, so $\Ext^2(V,W) = 0$.  By Corollary \ref{kroneckerOrthogonal} we see $\Hom(V,W) = 0$.  One easily calculates $\chi(V,W) = 0$ using the additivity of the Euler characteristic, so $\Ext^1(V,W)=0$ follows.
\end{proof}

\section{The effective cone of the Hilbert scheme of points}\label{effConeSection}

We now combine our results on the resolution of ideal sheaves $I_Z$ with the orthogonality of Kronecker modules to construct extremal effective divisors on the Hilbert scheme of points $\P^{2[n]}$.

Consider a general ideal sheaf $I_Z$ of $n$ points.  Let $\mu$ be the minimum slope of a stable bundle with $\chi/r = n$, and assume $\mu$ is not exceptional.  In most cases, Theorem \ref{resTheorem} associates to $I_Z$ a stable bundle $W$.  Either $W$ or its dual admits a resolution by a pair of semi-exceptional bundles.  Such a resolution induces a stable Kronecker module $e$ as in Subsection \ref{ss-orth}, and the isomorphism class of this Kronecker module depends only on $W$.  In cases where the exact sequence of Theorem \ref{resTheorem} must be interpreted in the derived category instead, it is still the case that the complex $W^{\bullet}$ corresponds to a stable Kronecker module.  

Thus, so long as $\mu$ is non-exceptional, we obtain a dominant rational map $$\pi : \P^{2[n]} \dashrightarrow Kr(\udim e),$$ where $e$ is a Kronecker module corresponding to $W$.  In case $\mu$ is exceptional, the general $I_Z$ is already the cokernel of a map of semi-exceptional bundles, and this map can be regarded as a Kronecker module.  The Hilbert scheme therefore always admits a rational map to a suitable moduli space of semistable Kronecker modules.

It is clear that when $\mu$ is exceptional the map $\pi$ is birational.  In the general case, the map has positive-dimensional fibers.

\begin{lemma}
If $\mu$ is non-exceptional, then $\dim Kr(\udim e) < 2n$.  Thus the general fiber of the rational map $\P^{2[n]}\dashrightarrow Kr(\udim e)$ is positive-dimensional.

\end{lemma}
\begin{proof}
Let $\alpha$ be the exceptional slope associated to $\mu$, and suppose $\mu< \alpha$.  Let $V$ be the associated orthogonal bundle of Proposition \ref{orthoResProp}.  Then $V$ has height zero, so there is a Kronecker module $f$ giving rise to $V$, and $M(\ch V) \cong Kr(\udim f)$.  In Drezet \cite{Drezet} it is shown that there is a natural isomorphism $Kr(\udim f) \cong Kr (\udim e)$.  Thus it will be enough to show $\dim(M(\ch(V)))< 2n$.  The same reduction works in case $\mu>\alpha$.

Recall that $M(\ch(V))$ has dimension $r(V)^2(2\Delta(V)-1)+1$.  We have $\gamma(\mu) = n$ and $\Delta(V) = \delta(\mu)$, so we must show $$r(V)^2(2\delta(\mu)-1)+1< 2\gamma(\mu).$$  We check this inequality holds for $\mu \in (\alpha-x_\alpha,\alpha]$.  We have $r(V) = \alpha r_{\alpha}$.  The left-hand side is a convex function of $\mu$, while the right-hand side is linear in $\mu$.  Thus it suffices to check the inequality holds at the endpoints.  We have $\delta(\alpha-x_\alpha)=1/2$, while $$2\gamma(\alpha-x_\alpha) = 1+3(\alpha-x_\alpha)+(\alpha-x_\alpha)^2,$$ so the inequality holds at $\alpha-x_\alpha$.  At $\alpha$, we have $$r(V)^2(2\delta(\alpha)-1)+1 = (\alpha r_\alpha)^2(2(P(0)-\Delta_\alpha)-1)+1=(\alpha r_\alpha)^2\cdot \frac{1}{r_\alpha^2}+1=\alpha^2+1$$ while $$2\gamma(\alpha) = \alpha^2+3\alpha+1-\frac{1}{r_\alpha^2}.$$ As $n\geq 2$ we have $\alpha\geq 1$, so the required inequality follows.

We also must check the inequality holds when $\mu\in (\alpha,\alpha+x_\alpha)$; here things are slightly trickier.  We have $r(V) = (\alpha+3)r_\alpha$, and we must verify \begin{equation}\label{moduliIneq}((\alpha+3)r_\alpha)^2(2\delta(\mu)-1)+1 < 2\gamma(\mu).\end{equation} The issue is that this inequality does \emph{not} hold when substituting $\mu = \alpha$ (although it does still hold for $\mu = \alpha+x_\alpha$, as $\delta(\alpha+x_\alpha)=1/2$).  However, we have $$\mu - \alpha = \frac{c_1(V)}{(\alpha+3)r_\alpha} - \alpha = \frac{c_1(V)r_\alpha-(\alpha+3)\alpha r_\alpha^2}{(\alpha+3)r_\alpha^2}. $$ The numerator and denominator of this last fraction are integers, so since $\mu\neq \alpha$ we may assume $\mu \geq \mu_0 := \alpha + ((\alpha+3)r_\alpha^2)^{-1}$.  But in fact, plugging in $\mu = \mu_0$ to inequality (\ref{moduliIneq}) yields an \emph{equality}, so the convexity argument shows the inequality holds when $\mu \in (\mu_0,\alpha+x_\alpha)$; we must rule out the possibility that $\mu = \mu_0$ can actually occur.   So suppose $\mu = \mu_0$.  We have $$2n = 2\gamma(\mu_0) = \alpha^2 +3\alpha+1+\frac{1}{r_\alpha^2} = 2\cdot \frac{\chi_\alpha}{r_\alpha},$$ which means that in fact we must have $\mu = \alpha$, as $\chi_\alpha/r_\alpha = n$.  Thus this case never actually arises,  and the required inequality holds.
\end{proof}

\begin{theorem}
Let $\mu$ be the minimum slope of a stable vector bundle on $\P^2$ having the property $\chi/r = n$.  Let $V$ be a general stable bundle of slope $\mu$ with $\chi/r = n$ such that $r$ is sufficiently large and divisible.  Then $V$ has interpolation for $n$ points, and the effective divisor $D_V(n)$ is extremal.  The effective cone of $\P^{2[n]}$ is spanned by $$\mu H - \frac{1}{2}\Delta \qquad \textrm{and} \qquad \Delta.$$
\end{theorem}
\begin{proof}
Let $\alpha.\beta$ be the exceptional slope associated to $\mu$, as in Section \ref{resolutionSection}, and let $Z\in \P^{2[n]}$ be general. First assume $\mu<\alpha.\beta$.  We assume $\chi(E_{\alpha.\beta}\te I_Z) r_{\alpha.\beta} \geq 3$; the details in the other ``derived" cases are essentially the same. We have a resolution $$0\to W \to E_{-(\alpha.\beta)}\te \Hom(E_{-(\alpha.\beta)},I_Z)\to I_Z\to 0$$ as in Theorem \ref{resTheorem}.  With $m_1$, $m_2$ as in the theorem, let $V$ be a general bundle with resolution $$0\to E_{\beta-3}^{km_1}\to E_{\alpha}^{km_2}\to V\to 0,$$ where $k$ is a sufficiently large integer.  Since $E_{\beta-3}\te E_{-(\alpha.\beta)}$ and $E_{\alpha}\te E_{-(\alpha.\beta)}$ are acyclic, $V\te E_{-(\alpha.\beta)}$ is acyclic.  Thus it suffices to show $V\te W$ is acyclic.  Now $W$ has a resolution $$0 \to W \to E_{-\alpha-3}^{m_1}\to E_{-\beta}^{3r_{\alpha.\beta} m_1-m_2}\to 0.$$  To show $V\te W$ is acyclic it suffices to show $\sHom(V,W^*(-3))$ is acyclic.  But $W^*(-3)$ has a resolution $$0\to E_{\beta-3}^{3r_{\alpha.\beta} m_1-m_2}\to E_{\alpha}^{m_1}\to W^*(-3)\to 0.$$ Writing  $N = 3 r_{\alpha.\beta} = \dim \Hom(E_{\beta-3},E_{\alpha})$, $a = m_1$, and $b = Nm_1-m_2$, we have $Na-b = m_2$, so we see Corollary \ref{quotientOrthogonal} implies the acyclicity of $\sHom(V,W^*(-3))$ as soon as we show $b/a\in (\psi_N^{-1},\psi_N)$.  But simple algebra shows the inequality $$\frac{m_1}{m_2} > r_{\alpha.\beta} x_{\alpha.\beta} = \frac{N}{3}\left(\frac{3}{2}-\sqrt{\frac{9}{4}-\frac{9}{N^2}}\right)$$ from Proposition \ref{orthoResProp} is equivalent to the inequality $$\frac{b}{a} = N-\frac{m_2}{m_1} > \psi_{N}^{-1};$$ the other needed inequality is trivial to establish.

To see that $D_V(n)$ is extremal, observe that it is the pullback of a divisor $D_f$ on $Kr(\udim e)$ under the rational map $\pi$, where $f$ is the Kronecker module corresponding to the resolution of $V$ (see Remark \ref{KroneckerDivisor}).  

If $\mu>\alpha.\beta$ an identical argument works.  In case $\mu=\alpha.\beta$, interpolation follows from Theorem \ref{exceptionalInterpThm}.  To see the divisor is extremal, recall that the general $I_Z$ has a resolution $$0\to E_{-\alpha-3}\te \Ext^1(I_Z,E_{-\alpha-3})^*\to E_{-\beta} \te \Hom(E_{-\beta},I_Z)\to I_Z\to 0.$$ Any $I_Z$ with a resolution of this form has $V\te I_Z$ acyclic.  By Proposition \ref{bertiniProp} and the methods of Section \ref{resolutionSection}, we can vary the map in the resolution to produce complete curves in the Hilbert scheme consisting entirely of schemes admitting resolutions as above.  This gives a moving curve class dual to $D_V(n)$, so the divisor is extremal. 
\end{proof}

\begin{remark}
Even in the general case, one can produce moving curves on the Hilbert scheme dual to the extremal effective divisor as in the final part of the proof of the theorem.  For instance, starting from a resolution of the form $$ 0 \to W \fto\psi E_{-(\alpha.\beta)}^{m_3} \to I_Z,$$ one can vary the map $\psi$ to produce complete curves in fibers of the rational map $\pi$.

In case the exceptional slope $\alpha.\beta$ is an integer, it is easy to construct a moving curve classically.  Write $n = r(r+1)/2+s$, with $0\leq s\leq r$.  The assumption that $\alpha.\beta$ is an integer amounts to requiring either $s/r > \varphi^{-1}$ or $\frac{s+1}{r+2}<1-\varphi^{-1}$, where $\varphi$ is the golden ratio; these two inequalities correspond to the possibilities $\mu< \alpha.\beta$ and $\mu\geq \alpha.\beta$, respectively.  In the former case, there is a dual moving curve given by letting $n$ points move in a linear pencil on a smooth curve of degree $r$; in the latter case, we get a dual moving curve by letting $n$ points move in a linear pencil on a smooth curve of degree $r+2$.  See \cite{ABCH,HuizengaPaper,thesis} for details.

We have also described a dual moving curve in case $\alpha.\beta$ is a half-integer $k/2$, with $k$ odd, and $\mu<\alpha.\beta$.  Writing $n$ as in the previous paragraph, this corresponds to the case where $$\sqrt{2}-1< \frac{s}{r-\frac{1}{2}} <\frac{1}{2}.$$ In this case, we showed in \cite{HuizengaPaper,thesis} that for a general collection $Z$ of $n$ points there is a curve $C$ of degree $2r-1$ which has $r^2-(r-1)-n$ nodes and no further singularities, such that $Z$ moves in a linear pencil on $C$. Allowing $Z$ to move in such a linear pencil describes a moving curve on $\P^{2[n]}$, and it is dual to the extremal divisor $D_V(n)$.
\end{remark}

\begin{remark}
The map $i_q: \P^{2[n]}\dashrightarrow \P^{2[n+1]}$ given by taking the union of a scheme with a fixed point $q\in \P^2$ induces an isomorphism $\Pic(\P^{2[n]})\cong \Pic(\P^{2[n+1]})$ identifying the divisors $H$ and $\Delta$ in each space.  Up to this identification, there is a natural inclusion $\Eff(\P^{2[n+1]})\subset \Eff(\P^{2[n]})$.  Combining the theorem with the results of Section \ref{gammaSection} we see that this inclusion is strict unless $n$ is of the form ${r+2\choose 2}-1$, when both effective cones are spanned by $rH-\frac{1}{2}\Delta$ and $\Delta$.
\end{remark}

In Table \ref{edgeTable}, we explicitly compute the nontrivial edge of the effective cone of the Hilbert scheme $\P^{2[n]}$ for small $n$.  This data can be generated very quickly by computer using the results of Section \ref{gammaSection}.  Remark \ref{xiRemark} is especially useful for this.

\begin{table}
\begin{center}
\caption{For each $n\geq 2$, the nontrivial edge of $\Eff \P^{2[n]}$ is spanned by $\mu H - \frac{1}{2}\Delta.$  The associated exceptional slope to $\mu$ is $\alpha$.}
\vspace{-.9em}
\setlength{\extrarowheight}{-1in}
\begin{tabular}{ccccccccccccccc}\toprule
$n$ & $\alpha$ & $\mu$ &\hphantom{+}\vphantom{I}& $n$ & $\alpha$ & $\mu$&\hphantom{+}& $n$ & $\alpha$ & $\mu$ &\hphantom{+}& $n$ & $\alpha$ & $\mu$\\\midrule
2 & 1 & 1 && 45 & 8 & 8 && 88 & 12 & 71/6 && 131 & 15 & 221/15 \\
3 & 1 & 1 && 46 & 8 & 90/11 && 89 & 12 & 143/12 && 132 & 15 & 74/5 \\
4 & 3/2 & 3/2 && 47 & 8 & 91/11 && 90 & 12 & 12 && 133 & 15 & 223/15 \\
5 & 2 & 2 && 48 & 8 & 92/11 && 91 & 12 & 12 && 134 & 15 & 224/15 \\
6 & 2 & 2 && 49 & 17/2 & 144/17 && 92 & 12 & 182/15 && 135 & 15 & 15 \\
7 & 12/5 & 12/5 && 50 & 17/2 & 197/23 && 93 & 12 & 61/5 && 136 & 15 & 15 \\
8 & 3 & 8/3 && 51 & 9 & 26/3 && 94 & 12 & 184/15 && 137 & 15 & 136/9 \\
9 & 3 & 3 && 52 & 9 & 79/9 && 95 & 12 & 37/3 && 138 & 15 & 91/6 \\
10 & 3 & 3 && 53 & 9 & 80/9 && 96 & 62/5 & 62/5 && 139 & 15 & 137/9 \\
11 & 3 & 10/3 && 54 & 9 & 9 && 97 & 25/2 & 312/25 && 140 & 15 & 275/18 \\
12 & 7/2 & 7/2 && 55 & 9 & 9 && 98 & 25/2 & 389/31 && 141 & 15 & 46/3 \\
13 & 4 & 15/4 && 56 & 9 & 55/6 && 99 & 164/13 & 164/13 && 142 & 77/5 & 1185/77 \\
14 & 4 & 4 && 57 & 9 & 37/4 && 100 & 13 & 165/13 && 143 & 31/2 & 479/31 \\
15 & 4 & 4 && 58 & 9 & 28/3 && 101 & 13 & 166/13 && 144 & 31/2 & 31/2 \\
16 & 4 & 30/7 && 59 & 19/2 & 179/19 && 102 & 13 & 167/13 && 145 & 31/2 & 576/37 \\
17 & 9/2 & 40/9 && 60 & 19/2 & 19/2 && 103 & 13 & 168/13 && 146 & 16 & 125/8 \\
18 & 23/5 & 23/5 && 61 & 48/5 & 48/5 && 104 & 13 & 13 && 147 & 16 & 251/16 \\
19 & 5 & 24/5 && 62 & 10 & 97/10 && 105 & 13 & 13 && 148 & 16 & 63/4 \\
20 & 5 & 5 && 63 & 10 & 49/5 && 106 & 13 & 105/8 && 149 & 16 & 253/16 \\
21 & 5 & 5 && 64 & 10 & 99/10 && 107 & 13 & 211/16 && 150 & 16 & 127/8 \\
22 & 5 & 21/4 && 65 & 10 & 10 && 108 & 13 & 53/4 && 151 & 16 & 255/16 \\
23 & 5 & 43/8 && 66 & 10 & 10 && 109 & 13 & 213/16 && 152 & 16 & 16 \\
24 & 11/2 & 11/2 && 67 & 10 & 132/13 && 110 & 13 & 107/8 && 153 & 16 & 16 \\
25 & 6 & 17/3 && 68 & 10 & 133/13 && 111 & 27/2 & 121/9 && 154 & 16 & 306/19 \\
26 & 6 & 35/6 && 69 & 10 & 134/13 && 112 & 27/2 & 27/2 && 155 & 16 & 307/19 \\
27 & 6 & 6 && 70 & 135/13 & 135/13 && 113 & 27/2 & 448/33 && 156 & 16 & 308/19 \\
28 & 6 & 6 && 71 & 21/2 & 220/21 && 114 & 14 & 191/14 && 157 & 16 & 309/19 \\
29 & 6 & 56/9 && 72 & 21/2 & 95/9 && 115 & 14 & 96/7 && 158 & 16 & 310/19 \\
30 & 6 & 19/3 && 73 & 11 & 117/11 && 116 & 14 & 193/14 && 159 & 16 & 311/19 \\
31 & 13/2 & 84/13 && 74 & 11 & 118/11 && 117 & 14 & 97/7 && 160 & 33/2 & 542/33 \\
32 & 13/2 & 125/19 && 75 & 11 & 119/11 && 118 & 14 & 195/14 && 161 & 33/2 & 544/33 \\
33 & 7 & 47/7 && 76 & 11 & 120/11 && 119 & 14 & 14 && 162 & 33/2 & 215/13 \\
34 & 7 & 48/7 && 77 & 11 & 11 && 120 & 14 & 14 && 163 & 83/5 & 1377/83 \\
35 & 7 & 7 && 78 & 11 & 11 && 121 & 14 & 240/17 && 164 & 17 & 283/17 \\
36 & 7 & 7 && 79 & 11 & 78/7 && 122 & 14 & 241/17 && 165 & 17 & 284/17 \\
37 & 7 & 36/5 && 80 & 11 & 157/14 && 123 & 14 & 242/17 && 166 & 17 & 285/17 \\
38 & 7 & 73/10 && 81 & 11 & 79/7 && 124 & 14 & 243/17 && 167 & 17 & 286/17 \\
39 & 37/5 & 37/5 && 82 & 11 & 159/14 && 125 & 14 & 244/17 && 168 & 17 & 287/17 \\
40 & 15/2 & 15/2 && 83 & 23/2 & 263/23 && 126 & 418/29 & 418/29 && 169 & 17 & 288/17 \\
41 & 8 & 61/8 && 84 & 23/2 & 23/2 && 127 & 29/2 & 420/29 && 170 & 17 & 17 \\
42 & 8 & 31/4 && 85 & 336/29 & 336/29 && 128 & 29/2 & 509/35 && 171 & 17 & 17 \\
43 & 8 & 63/8 && 86 & 12 & 35/3 && 129 & 73/5 & 73/5 &&  &  &   \\
44 & 8 & 8 && 87 & 12 & 47/4 && 130 & 15 & 44/3 &&  &  &   \vspace{-.3em} \\ \bottomrule
\end{tabular}\label{edgeTable}\end{center}\end{table}

\section{Connections with Bridgeland stability}\label{bridgelandSection}

In \cite{ABCH}, it was conjectured that there is a correspondence between the walls in the Mori chamber decomposition of the Hilbert scheme $\P^{2[n]}$ and the walls in a suitable half-plane of Bridgeland stability conditions. Our goal for the rest of the article is to show that our computation of the effective cone of $\P^{2[n]}$ is consistent with this conjecture.  The key step is to determine when exceptional bundles are Bridgeland semistable.

\subsection{Preliminaries on Bridgeland stability} We briefly summarize the necessary material from \cite{ABCH}; we refer the reader to sections 5-9 of that paper for a full account.  Let $D^b(\P^2)=D^b(\coh \P^2)$ be the bounded derived category of coherent sheaves on $\P^2$.  For any $s\in \R$, we define full subcategories $\cF_s$ and $\cQ_s$ of $\coh(\P^2)$ by the requirements
\begin{itemize}
\item $Q\in \cQ_s$ if and only if $Q$ is torsion, or every quotient in the Harder-Narasimhan filtration of $Q$ has slope larger than $s$.
\item $F\in \cF_s$ if and only if $F$ is torsion-free, and each quotient in the Harder-Narasimhan filtration of $F$ has slope no larger than $s$.
\end{itemize}
The subcategories $(\cF_s, \cQ_s)$ define a \emph{torsion pair} for each $s$.  Associated to this torsion pair is a corresponding $t$-structure on $D^b(\P^2)$.  The heart of this $t$-structure is the full abelian subcategory $\cA_s$ of $D^b(\P^2)$ given by complexes whose $H^{-1}$-term is in $\cF_s$ and whose $H^0$-term is in $\cQ_s$, with all other cohomology sheaves equal to zero:
$$\cA_s = \{ E^{\bullet}:H^{-1}(E^{\bullet}) \in \cF_s,\,H^0(E^{\bullet})\in \cQ_s,\,\textrm{and } H^{i}(E^\bullet)=0 \textrm{ for other $i$}\}.$$

Next we define on the category $\cA_s$ a family of slope functions; these will depend only on the Chern character $(r,c,d)= (\ch_0,\ch_1,\ch_2)$ of a complex $E^{\bullet}$.  For each real number $t>0$, put $$\mu_{s,t}(r,c,d) = \frac{-\frac{t^2}{2}r+(d-sc+\frac{s^2}{2}r)}{t(c-sr)}.$$ Then the pair $\cA_{s,t} = (\cA_s,\mu_{s,t})$ forms a \emph{Bridgeland stability condition} \cite{Bridgeland,ArcaraBertram,BayerMacri}.  One defines slope semistability of objects of $\cA_{s,t}$ in the obvious way.  For any Chern character and choice of $(s,t)$, the moduli space $\cM_{\P^2}(r,c,d)$ of semi-stable objects of $\cA_{s,t}$ with given Chern character can be constructed as an Artin stack \cite{AbramPol,ArcaraBertram,lieblich,toda}.  These spaces can also be constructed as projective schemes using geometric invariant theory \cite{ABCH,BayerMacri2}.

When $E\in \cQ_s$ is a coherent sheaf, we regard it as an object of $\cA_s$ by viewing it as a $0$-complex.  The following fact from \cite{ABCH} is particularly relevant to the present discussion, so we single it out.  The analogous fact for $K3$-surfaces was originally shown by Bridgeland \cite{Bridgeland}.

\begin{proposition}
Suppose $E\in \cQ_s$ is a Mumford-semistable sheaf.  There is a number $t_0>0$ such that $E$ is a $\mu_{s,t}$-semi-stable object of $\cA_s$ for all $t>t_0$.  Furthermore, there exists a uniform choice of $t_0$ depending only on the Chern character of $E$.
\end{proposition}

Conversely, it can be seen that if $E^\bullet\in \cA_s$ has $\ch_0(E^\bullet)\geq 0$ and $E^\bullet$ is $(s,t)$-semi-stable for large $t$, then in fact $H^0(E^\bullet)\in \cQ_s$ is a Mumford-semistable sheaf and $H^{-1}(E^\bullet)=0$.  Thus $E^\bullet$ is isomorphic to a Mumford-semistable sheaf in $\cQ_s$.  We conclude that if $s< c/r$ and $t \gg 0$ then the moduli space of semistable objects of $\cA_{s,t}$ with Chern character $(r,c,d)$ is just the ordinary moduli space of Mumford-semistable coherent sheaves.  In particular, if $s<0$ and $t\gg 0$, the moduli space of $(s,t)$-semistable objects of $\cA_s$ with Chern character $(1,0,-n)$ is isomorphic to $\P^{2[n]}$ for large $t$.

To understand the birational geometry of $\P^{2[n]}$, we study the problem of understanding how the moduli space of $(s,t)$-semistable objects of $\cA_s$ with Chern character $(1,0,-n)$ varies as $(s,t)$ varies in the quadrant  $\{s<0,t>0\}$.  When the collection of semistable objects changes, it is due to $(s,t)$ crossing a \emph{potential wall} where some ideal sheaf $I_Z$ is destabilized.  Precisely, for two Chern characters $(r,c,d)$ and $(r',c',d')$, the corresponding potential wall is the subset $$W_{(r,c,d),(r',c',d')} = \{(s,t): \mu_{s,t}(r,c,d) = \mu_{s,t}(r',c',d')\}.$$ When $E$ is a bundle with Chern character $(r,c,d)$, we frequently write $W_{E,(r',c',d')}$ for the preceding wall, and similarly with the second argument.  
If $E'\to E$ is an inclusion of objects of $\cA_s$, then $E$ is potentially destabilized as $(s,t)$ crosses the wall $W_{E,E'}$.  In fact, elementary calculus shows that if $(r,c,d)$ and $(r',c',d')$ are not proportional then on one side of the wall we have $\mu_{s,t}(E')>\mu_{s,t}(E)$ and on the other we have $\mu_{s,t}(E)>\mu_{s,t}(E')$.  In particular, $E$ cannot be semistable on both sides of the wall, but it could potentially be unstable on both sides.

The geometry of the walls we must consider is particularly nice.  Fix a Mumford-stable sheaf $E$ with Chern character $(r,c,d)$, and consider the family of walls $W_{E,(r',c',d')}$ as $(r',c',d')$ varies.  One wall is the vertical line $s = c/r = \mu(E)$, corresponding to $(r',c',d')$ with the same slope $c'/r' = \mu(E)$.  The other walls form two nested families of semicircles on either side of the vertical wall, with each semicircle centered on the $s$-axis in the $st$-plane.  The center is positioned at the point $$\left(\frac{rd'-r'd}{rc'-r'c},0\right),$$ and it has radius $$\sqrt{\left(\frac{r d'-r'd}{rc'-r'c}\right)^2-2\left(\frac{cd'-c'd}{rc'-r'c}\right)}.$$

The formula for the radius of a wall becomes more transparent when one looks at a Mumford-stable sheaf $E$ of slope $c/r=0$.  In this case, the Bogomolov inequality gives $d\leq 0$.  If we let $$x= \frac{rd'-r'd}{r c'},$$ then the wall $W_{E,(r',c',d')}$ has center $(x,0)$ and radius $$\sqrt{x^2+\frac{2 d}{r}} \leq |x|.$$ Noting that $2d/r$ is a fixed nonpositive number, we observe the following basic fact.  

\begin{lemma}\label{nestingByCenters} The radius of a semicircular wall to the left of the vertical wall decreases as the center moves to the right, toward the vertical wall.  Similarly, the radius of a wall to the right of the vertical wall decreases as the center moves to the left. 
\end{lemma}

The restriction that $\mu(E)=0$ in the preceding discussion is not essential; formally twisting by $-\mu(E)$ shifts all the walls by $\mu(E)$, so the general case follows from this.  Thus in order to show one wall $W_{E,(r',c',d')}$ is nested in another $W_{E,(r'',c'',d'')}$ it is enough to show both walls lie on the same side of the vertical wall and that the center of the first wall is closer to the vertical wall than the center of the second wall is.  This fact can be useful, as the expression for the radius is far more complicated than the expression for the center.

If $E$ is a Mumford-stable sheaf in $\cQ_{s_0}$, then we say an injection $F\to E$ in the category $\cA_{s_0}$ \emph{destabilizes} $E$ at $(s_0,t_0)$ if $E$ is $(s_0,t_0)$-semistable and $(s_0,t_0)$ lies on the wall $W_{E,F}$, so that $E,F$ have the same $(s_0,t_0)$-slope.  In this case, for every $(s,t)\in W_{E,F}$ the map $F\to E$ destabilizes $E$ at $(s,t)$; in particular, $E$ and $F$ are in $\cA_s$.  Thus we say $F$ destabilizes $E$ along the wall $W_{E,F}$.  Recall that $\mu_{s,t}(F)>\mu_{s,t}(E)$ for all $(s,t)$ on one side of the wall and $\mu_{s,t}(F)<\mu_{s,t}(E)$ for all $(s,t)$ on the other side of the wall; since $E$ is Mumford-stable we see that in fact $E$ is $(s,t)$-stable for all $(s,t)$ outside of the wall and $E$ is not $(s,t)$-semistable for any $(s,t)$ inside the wall.

We may now describe the conjectural correspondence between Bridgeland and Mori walls for $\P^{2[n]}$ discussed in \cite{ABCH}.  Let $I_Z$ be the ideal sheaf of some $Z\in \P^{2[n]}$, and consider the family of walls $W_{I_Z,(r',c',d')}$.  We call a wall in this family a \emph{Bridgeland wall} for the Hilbert scheme if some ideal sheaf $I_Z$ is destabilized along the wall.  With $n$ fixed, Bridgeland walls only depend on their centers, so we denote the Bridgeland wall with center $(x,0)$ by $W_x$.  On the other hand, a Mori wall is a ray $H + \frac{1}{2y}\Delta$ in $\Eff \P^{2[n]}$ where the stable base locus of a divisor in the ray changes; such walls depend only on the parameter $y<0$.

\begin{conjecture}
There is a one-to-one correspondence between Bridgeland walls $W_x$ and Mori walls $H+\frac{1}{2y}\Delta$ for $\P^{2[n]}$ given by the transformation $$x = y - \frac{3}{2}.$$
\end{conjecture}

The \emph{collapsing wall} is the Bridgeland wall where the general ideal sheaf $I_Z$ is destabilized; it is the innermost Bridgeland wall.  The proof of the next theorem will occupy the rest of this section.

\begin{theorem}\label{bridgelandThmUnrefined}
The center $(x,0)$ of the collapsing wall for $\P^{2[n]}$ corresponds to the nontrivial edge $\mu H - \frac{1}{2}\Delta$ of $\Eff(\P^{2[n]})$ by $$x = -\left(\mu+\frac{3}{2}\right).$$
\end{theorem}

\subsection{The destabilizing object for a general ideal sheaf}

To prove Theorem \ref{bridgelandThmUnrefined}, we need to identify the collapsing wall and verify the relation between its center and the edge of the effective cone.  It is relatively easy to specify what the collapsing wall is; the difficult part is to show that the general ideal sheaf $I_Z$ is actually semistable along the collapsing wall.  Here we describe the collapsing wall, and leave the proof of semistability of the general ideal sheaf for the next section.

Let $\mu$ be the minimum slope of a stable bundle $V$ with $\chi/r=n$, and let $\alpha.\beta$ be the associated exceptional slope to $\mu$, as in Section \ref{resolutionSection}.  Assume for now that $\mu<\alpha.\beta$. Let $I_Z$ be a general ideal sheaf of $n$ points. By Theorem \ref{resTheorem} there is a distinguished triangle $$W^{\bullet}\to E_{-(\alpha.\beta)} \te \Hom(E_{-(\alpha.\beta)},I_Z)\to I_Z \to$$ where $W^\bullet$ is the complex $$E_{-\alpha-3}^{m_1}\to E_{-\beta}^{3r_{\alpha.\beta}m_1-m_2}.$$ The shift $$E_{-(\alpha.\beta)}\te \Hom(E_{-(\alpha.\beta),I_Z})\to I_Z\to W^{\bullet}[1]\to$$ is then also a distinguished triangle.  In case $\chi(E_{\alpha.\beta}\te I_Z)r_{\alpha.\beta}\geq 3$, so that $W = W^\bullet$ is actually a stable vector bundle, we observe that $\mu(W) < \mu(E_{-(\alpha.\beta)})$ since $c_1(W) = c_1(E_{-(\alpha.\beta)})<0$ and $\rk(W) = \rk(E_{-(\alpha.\beta)})-1.$  Thus if $\mu(W) < s < \mu(E_{-(\alpha.\beta)})$ we have $E_{-(\alpha.\beta)}\in \cQ_s$ and $W \in \cF_s$, so $W[1]\in \cA_s$.  Thus all the terms in the above triangle are in $\cA_s$, and we have an exact sequence $$0\to E_{-(\alpha.\beta)}\te \Hom(E_{-(\alpha.\beta),I_Z})\to I_Z\to W[1]\to 0$$ in the category $\cA_s$. The cases with $\chi(E_{\alpha.\beta}\te I_Z)r_{\alpha.\beta}\leq 2$ are easily handled case by case, and we arrive at the same exact sequence if $W$ is interpreted as a complex.  Treating the cases where $\mu> \alpha.\beta$ and $\mu = \alpha.\beta$ similarly, it is natural to expect the following theorem is true.

\begin{theorem}\label{bridgelandThm}
Let $Z\in \P^{2[n]}$ be general.  When $\mu<\alpha.\beta$ or $n$ is of the form ${r+2\choose 2}-1$, the canonical homomorphism $$E_{-(\alpha.\beta)}\te \Hom(E_{-(\alpha.\beta)},I_Z)\to I_Z$$ is a destabilizing subobject of $I_Z$.  

If $\mu > \alpha.\beta$, the canonical homomorphism $$I_Z \to E_{-(\alpha.\beta)-3}[1]\te \Ext^{1}(I_Z,E_{-(\alpha.\beta)-3})^*$$ is a destabilizing quotient object of $I_Z$.  In other words, this map is surjective in appropriate categories $\cA_s$, and its kernel $W$ is a destabilizing subobject.

If $\mu = \alpha.\beta$ and $n$ is not of the form ${r+2\choose 2}-1$, the canonical homomorphisms $$E_{-\beta}\te \Hom(E_{-\beta},I_Z)\to I_Z$$ and $$I_Z\to E_{-\alpha-3}[1]\te \Ext^1(I_Z,E_{-\alpha-3})^*$$ are destabilizing sub- and quotient objects of $I_Z$, respectively.

In each case, the center $(x,0)$ of the corresponding wall is given by $x = -(\mu+\frac{3}{2})$.
\end{theorem}

Proving even one of the three cases of the theorem takes a large amount of calculation, so we focus exclusively on the case $\mu<\alpha.\beta$, and further assume $\chi(E_{\alpha.\beta}\te I_Z)r_{\alpha.\beta}\geq 3$.  Verifying the other cases would be a good exercise to become comfortable with the arithmetic of exceptional slopes and Bridgeland stability.  Let us point out the following general fact before beginning the proof:  given any exact sequence of sheaves $$0\to A\to B\to C\to 0,$$ there is an equality $W_{A,B}=W_{B,C}=W_{C,A}$.  These several descriptions of a given wall are frequently useful.

\begin{proof}[Proof for $\mu<\alpha.\beta$ and $\chi(E_{\alpha.\beta}\te I_Z)r_{\alpha.\beta}\geq 3$]
Consider the exact sequence $$0\to E_{-(\alpha.\beta)}\te \Hom(E_{-(\alpha.\beta)},I_Z)\to I_Z\to W[1]\to 0,$$ valid in any category $\cA_s$ with $\mu(W)<s<\mu(E_{-(\alpha.\beta)})$.  We must show the following facts:
 \begin{enumerate}
 \item The wall $W_{I_Z,E_{-(\alpha.\beta)}}=W_{I_Z,W}=W_{E_{-(\alpha.\beta)},W}$ is nonempty and lies between the vertical lines $s = \mu(W)$ and $s= -(\alpha.\beta)$.

\item The sheaf $E_{-(\alpha.\beta)}$ is $(s,t)$-semistable along the wall $W_{I_Z,E_{-(\alpha.\beta)}}$.

\item The object $W[1]$ is $(s,t)$-semistable along the wall $W_{I_Z,W}$.
 \end{enumerate}
 
 From (2) and (3) it follows that $I_Z$ is semistable along the wall, since it is an extension of semistable objects of the same slope.
 
Let us prove (1).  Since walls $W_{(r,c,d),E_{-(\alpha.\beta)}}$ fall into two nested families of semicircles on either side of the vertical wall $s = -(\alpha.\beta)$, to show that the wall $W_{I_Z,E_{-(\alpha.\beta)}}$ lies to the left of $s= -(\alpha.\beta)$ it is enough to show that its center lies to the left of this line (provided it is nonempty).  For any exceptional slope $\alpha\in \F E$ we have $$\ch(E_\alpha) = \left(r_\alpha, r_\alpha \alpha, r_{\alpha}\left(\frac{1}{2} \alpha^2-\Delta_\alpha\right)\right),$$ so the center of the wall $W_{I_Z,E_{-(\alpha.\beta)}}$ is positioned at the point $(x,0)$ with $$x = \frac{\ch_2(E_{-(\alpha.\beta)})+\ch_0(E_{-(\alpha.\beta)})n}{\ch_1(E_{-(\alpha.\beta)})}=-\frac{\alpha.\beta}{2}+\frac{\Delta_{\alpha.\beta}}{\alpha.\beta}-\frac{n}{\alpha.\beta} = -\left(\mu + \frac{3}{2}\right),$$ using the fact that $$n = \gamma(\mu) = (\alpha.\beta)(\mu+3)+1+\Delta_{\alpha.\beta}-P(\alpha.\beta).$$
 Thus the statement that $x<-(\alpha.\beta)$ is equivalent to the inequality
$\mu > \alpha.\beta - \frac{3}{2},$ which is obvious since $\alpha.\beta$ is the associated exceptional slope to $\mu$.  We conclude that if $W_{I_Z,E_{-(\alpha.\beta)}}$ is nonempty, it lies to the left of the vertical line $s = -(\alpha.\beta)$.  Furthermore, the distance between the center of $W_{I_Z,E_{-(\alpha.\beta)}}$ and this line is $-(\alpha.\beta)-x=\frac{3}{2}-\alpha.\beta+\mu<3/2$.

Likewise, to see the wall $W_{I_Z,W}$ lies to the right of the wall $s = \mu(W)$, it suffices to see $x > \mu(W)$. By our final observation in the previous paragraph, we can show the distance $\mu(E_{-(\alpha.\beta)})-\mu(W)$ between the two vertical walls exceeds $3/2$.  We have $$\mu(E_{-(\alpha.\beta)})-\mu(W) = -(\alpha.\beta) + \frac{m_3 r_{\alpha.\beta}(\alpha.\beta)}{m_3 r_{\alpha.\beta}-1}=\frac{\alpha.\beta}{m_3 r_{\alpha.\beta}-1}.$$ Since $$m_3 = \chi(E_{\alpha.\beta}\te I_Z) = \chi_{\alpha.\beta} - n r_{\alpha.\beta}= \chi_{\alpha.\beta}-\gamma(\mu)r_{\alpha.\beta}$$ we see this distance exceeds the number $$\frac{\alpha.\beta}{(\chi_{\alpha.\beta} - \gamma(\alpha.\beta - x_{\alpha.\beta})r_{\alpha.\beta})r_{\alpha.\beta}-1}= \frac{1}{x_{\alpha.\beta}r_{\alpha.\beta}^2}.$$ This final quantity depends only on the integer $r_{\alpha.\beta}$, and is an increasing function of $r_{\alpha.\beta}$.  Already for $r_{\alpha.\beta}=1$ it equals $\varphi + 1 \approx 2.618$, so $W_{I_Z,W}$ must lie to the right of $s= \mu(W)$ if it is nonempty.

Finally let us show this wall is actually nonempty.  The radius $\rho$ of this wall satisfies $$\rho^2 = x^2 - 2n = \left(\mu+\frac{3}{2}\right)^2-2\gamma(\mu)=2P(\mu)+\frac{1}{4}-2 \gamma(\mu)=2\delta(\mu)+\frac{1}{4}>\frac{5}{4}$$ since $\delta(\mu)>1/2$ for all $\mu\in \Q$.  Thus the radius is at least $\sqrt{5}/2$; in particular the wall is nonempty.

To complete the proof, it will be sufficient to show that $E_{-(\alpha.\beta)}$ and $W[1]$ are $(s,t)$-semistable outside of semicircular walls of radius at most $\sqrt{5}/2$.  We will prove this in the next section.
\end{proof}

\section{Bridgeland stability of exceptional bundles}\label{bridgelandSec2}

In this section we investigate the Bridgeland semistability of exceptional bundles in order to complete the proof of Theorem \ref{bridgelandThm}. Our main result ensures that the locus of $(s,t)$ where an exceptional bundle is not $(s,t)$-semistable is not ``too large.''

\begin{theorem}\label{excBridgeThm}
Let $\alpha,\beta$ be exceptional slopes of the form $$\alpha = \varepsilon\left(\frac{p\vphantom1}{2^q}\right) \qquad \beta = \varepsilon\left(\frac{p+1}{2^q}\right),$$ where $p$ is even.  The exceptional bundle $E_{\beta}$ is semistable along the semicircular wall $W_{E_{\alpha},E_{\beta}}$, and stable outside this wall.
\end{theorem}

We will see later that the wall $W_{E_\alpha,E_\beta}$ is nonempty and lies to the left of the vertical wall $s = \mu(E_\beta) = \beta$ so long as $q\geq 1$.  In case $\beta$ is an integer, we have $E_\beta = \OO_{\P^2}(\beta)$ and $E_\beta$ is $(s,t)$-stable for all $s<\beta$ by \cite[Proposition 6.2]{ABCH}.  We thus assume $q\geq 1$ for the rest of the section.

To prove the theorem it will be necessary to simultaneously address the stability of shifted exceptional bundles of the form $E_\beta[1]$.  The next lemma will allow us to treat these on an essentially equal footing with ordinary exceptional bundles.

\begin{lemma}
Let $E$ be a Mumford-stable sheaf, and suppose either $E\in \cQ_s$ or $E\in \cF_s$; write $E'=E$ in the first case and $E'=E[1]$ in the second case, so that $E'\in \cA_s$.  Then $E'$ is $(s,t)$-stable for sufficiently large $t$.  If $E'$ is $(s_0,t_0)$-semistable for some $(s_0,t_0)$, then $E'$ is $(s,t)$-stable whenever the semicircular wall passing through $(s_0,t_0)$ is nested inside the semicircular wall passing through $(s,t)$.
\end{lemma}

In \cite[Section 6]{ABCH} the case where $E\in \cQ_s$ was handled; when $E\in \cF_s$ similar methods can be used, so we omit the proof.  Note in particular that the walls for $E[1]$ are the same as the walls for $E$; however, while the relevant family of semicircles for a sheaf $E\in \cQ_s$ is the family to the left of the vertical wall $s = \mu(E)$, the relevant family of semicircles for $E[1]$ when $E\in \cF_s$ is the family to the \emph{right} of the vertical wall $s = \mu(E)$, as this is the region where $E[1]\in \cA_s$.

It will be useful to introduce some additional exceptional slopes.  We put $$\zeta_0 = \varepsilon\left(\frac{p+4}{2^q}-3\right)\qquad \zeta_2 = \varepsilon\left(\frac{p-2}{2^q}\right) $$$$\alpha = \varepsilon\left(\frac{p\vphantom1}{2^q}\right) \qquad \beta = \varepsilon\left(\frac{p+1}{2^q}\right) \qquad \eta = \varepsilon\left(\frac{p+2}{2^q}\right)$$$$ \quad \omega_0=\varepsilon\left(\frac{p+4}{2^q}\right)\qquad \omega_2=\varepsilon\left(\frac{p-2}{2^q}+3\right),$$ where $p$ is even and $q\geq 1$.  Observe that $\beta = \alpha.\eta$, if $p\equiv 2\pmod{4}$ then $\alpha =\zeta_2.\eta$, and if $p\equiv 0 \pmod 4$ then $\eta=\alpha.\omega_0$.  The significance of the slopes $\zeta_i$ and $\omega_i$ is provided by the following result from \cite{DrezetBeilinson}.

\begin{theorem}\label{kernelSlope}
Let $i\in \{0,2\}$ be such that $i\equiv p \pmod 4$.  There are exact sequences of vector bundles$$0\to E_{\zeta_i}\to E_\alpha\te \Hom(E_\alpha,E_\beta)\to E_\beta\to 0$$ and $$0\to E_\beta\to E_{\eta}\te \Hom(E_{\beta},E_\eta)^*\to E_{\omega_i}\to 0.$$
\end{theorem}

The theorem provides an inductive description for building up $E_\beta$ in terms of simpler exceptional bundles.  We record how these exact sequences interact with the categories $\cA_s$.

\begin{proposition}\label{exceptionalBridgelandSeqLemma}
For each of the following four cases, there exists some $(s,t)$ such that the displayed sequence is an exact sequence of objects of $\cA_s$ with the same $\mu_{s,t}$-slope.

\begin{enumerate}
\item For $p\equiv 0 \pmod 4$ or $q=1$, $$0\to E_{\alpha}\te \Hom(E_\alpha,E_\beta)\to E_\beta \to E_{\zeta_0}[1]\to 0.$$

\item For $p\equiv 2 \pmod 4$ and $q\geq 2$, $$0\to E_{\zeta_2}\to E_\alpha\te \Hom(E_{\alpha},E_\beta)\to E_\beta\to 0.$$

\item For $p\equiv 0 \pmod 4$ and $q\geq 2$, $$0\to E_\beta[1] \to E_\eta[1] \te \Hom(E_\beta,E_\eta)^*\to E_{\omega_0}[1]\to 0.$$

\item For $p\equiv 2 \pmod 4$ or $q=1$, $$0\to E_{\omega_2}\to E_{\beta}[1]\to E_{\eta}[1]\te \Hom(E_\beta,E_\eta)^*\to 0.$$
\end{enumerate}
\end{proposition}

Before proving the proposition let us isolate some numerical facts that will be useful.

\begin{lemma}\label{centerLemma}
Let $\alpha,\beta$ be two exceptional slopes as in this section, except do not require that $p$ be even (so that, in particular, this result also applies to the pair of slopes $(\beta,\eta)$ from this section).  The center of the wall $W_{E_\alpha,E_\beta}$ is located at the point $(x,0)$ with $$x = \frac{\alpha+\beta}{2}+ \frac{\Delta_\beta-\Delta_\alpha}{\alpha-\beta},$$ and the radius $\rho$ satisfies $$\rho^2 = \left(\frac{\alpha-\beta}{2}\right)^2-P(\alpha-\beta)+\left(\frac{\Delta_\beta-\Delta_\alpha}{\alpha-\beta}\right)^2.$$ In fact, the formula for the center is valid for \emph{any} two exceptional slopes (the formula for the radius is not).
\end{lemma}

The proof is a straightforward computation with the formulas of the previous section.  Quantities of the form $(\Delta_\beta-\Delta_\alpha)/(\alpha-\beta)$ occur frequently, and we must estimate them.

\begin{lemma}\label{centerIneqLemma}
With the notation of this section, suppose $q\geq 2$.   Then $$\frac{\Delta_\beta-\Delta_\alpha}{\alpha-\beta}<-1 \qquad \textrm{and}\qquad \frac{\Delta_\eta-\Delta_\beta}{\beta-\eta}>1.$$ When $q=1$, the first quantity equals $-3/4$ and the second is $3/4$.
\end{lemma}
\begin{proof}
We verify the first inequality.  It is equivalent to $$\alpha  -2\Delta_\alpha>\beta - \Delta_\beta -\Delta_\alpha = \beta - P(\alpha-\beta).$$ This can be rearranged to give $$\frac{1}{r_{\alpha}^2}> \frac{1}{2}(\beta-\alpha)(5+\alpha-\beta).$$ Since $\alpha<\beta$, it is enough to show $$\frac{5}{2}(\beta-\alpha)\leq \frac{1}{r_{\alpha}^2}.$$ Since $\beta = \alpha.\eta$, we find $$\beta-\alpha = \frac{1}{r_{\alpha}^2(3+\alpha-\eta)}.$$ As $q\geq 2$ we have $\eta-\alpha \leq 1/2$, and the required inequality follows.
\end{proof}

\begin{corollary}\label{nonemptyWall}
The walls $W_{E_\alpha,E_\beta}$ and $W_{E_{\beta},E_{\eta}}$ are nonempty.
\end{corollary}
\begin{proof}
Lemmas \ref{centerLemma} and \ref{centerIneqLemma} show the radius $\rho$ of each wall satisfies $\rho>0$.
\end{proof}

\begin{proof}[Proof of Proposition \ref{exceptionalBridgelandSeqLemma}]
In each case, the result amounts to showing that an appropriate (nonempty) potential wall lies in the region where all the objects in the exact sequence lie in the category $\cA_s$.

(1) Suppose $p\equiv 0 \pmod 4$.  We must show the wall $W_{E_\alpha,E_\beta}=W_{E_{\zeta_0},E_\beta}=W_{E_{\zeta_0},E_\alpha}$ lies in the strip $\{(s,t):\zeta_0<s < \alpha\},$ where the exact sequence is valid in $\cA_s$.  For this we may show the center of the wall lies in the strip.  Using the description of the wall as $W_{E_{\zeta_0},E_\alpha}$, we must verify the inequalities $$\zeta_0 <  \frac{\alpha+\zeta_0}{2}+\frac{\Delta_{\zeta_0}-\Delta_\alpha}{\alpha-\zeta_0}<\alpha.$$ But $\alpha-\zeta_0\geq 1$ (with equality for $q=1$), so $$\left| \frac{\Delta_{\zeta_0}-\Delta_\alpha}{\alpha-\zeta_0}\right|<\frac{1}{2},$$ and both inequalities hold.

(2) If $p\equiv 2 \pmod 4$ and $q\geq 2$, then noting $\zeta_2< \alpha< \beta$ we claim the wall $W_{E_\alpha,E_\beta} = W_{E_\alpha,E_{\zeta_2}}$ lies to the left of the vertical wall $s = \zeta_2$, i.e. that $$\frac{\alpha+\zeta_2}{2}+\frac{\Delta_{\zeta_2}-\Delta_\alpha}{\alpha-\zeta_2}<\zeta_2.$$ Recalling $\zeta_2 = \varepsilon((p-2)/2^q)$, we see that Lemma \ref{centerIneqLemma} gives $$\frac{\Delta_{\zeta_2} - \Delta_\alpha}{\alpha-\zeta_2} \leq -\frac{3}{4}$$ since $q\geq 2$. But $(\alpha+\zeta_2)/2-\zeta_2 = (\alpha-\zeta_2)/2\leq 1/4$, so the required inequality holds.

Cases (3) and (4) are mirror images of the previous two cases.\end{proof}

\begin{lemma}\label{nestingLemma}
The wall $W_{E_\alpha,E_\beta}$ is nested inside the wall $W_{E_\alpha,E_{\alpha.\beta}}$, and the wall $W_{E_\eta,E_\beta}$ is nested inside the wall $W_{E_\eta,E_{\beta.\eta}}$.
\end{lemma}
\begin{proof}
We check the first statement.  The center of the wall $W_{E_\alpha,E_\beta}$ is positioned at the point $(x,0)$ with $$x=\frac{\alpha+\beta}{2}+\frac{\Delta_\beta-\Delta_\alpha}{\alpha-\beta}.$$  We have $(\alpha+\beta)/2 - \alpha = (\beta-\alpha)/2 \leq 1/4$ since $q\geq 1$, so Lemma \ref{centerIneqLemma} shows $x<\alpha$ and thus the wall $W_{E_{\alpha},E_\beta}$ is centered to the left of the vertical wall $s = \alpha$.  To show it is nested inside $W_{E_\alpha,E_{\alpha.\beta}}$, we use Lemma \ref{nestingByCenters} and show the center of the latter wall lies to the left of the center of the former.  In symbols, we must establish the inequality $$\frac{\alpha+\alpha.\beta}{2}+\frac{\Delta_{\alpha.\beta}-\Delta_\alpha}{\alpha-\alpha.\beta}< \frac{\alpha+\beta}{2}+\frac{\Delta_\beta-\Delta_\alpha}{\alpha-\beta}.$$ Using the by now standard identities of Lemma \ref{numericalProps}, one easily shows this inequality is equivalent to the inequality $2\Delta_\beta <1.$

Note that the wall $W_{E_\eta,E_\beta}$ is located to the right of the wall $s = \eta$; otherwise the argument is identical.
\end{proof}

\begin{corollary}\label{root5Cor}
The walls $W_{E_\alpha,E_\beta}$ and $W_{E_\eta,E_\beta}$ have radius smaller than $\sqrt{5}/2$.
\end{corollary}
\begin{proof}
Consider the sequence of walls $$W_{E_\alpha,E_\beta}, W_{E_\alpha,E_{\alpha.\beta}}, W_{E_\alpha,E_{\alpha.(\alpha.\beta)}}, W_{E_\alpha,E_{\alpha.(\alpha.(\alpha.\beta))}},\ldots$$ with each wall nested inside the next.  The sequence $$\beta,\alpha.\beta,\alpha.(\alpha.\beta),\alpha.(\alpha.(\alpha.\beta)),\ldots$$ is decreasing and converges to $\alpha+x_\alpha$, and the discriminants of the corresponding exceptional bundles converge to $1/2$.  Thus the squares of the radii increase and converge to $$\left(\frac{x_\alpha}{2}\right)^2-P(-x_\alpha)+\left(\frac{\frac{1}{2}-\Delta_\alpha}{x_\alpha}\right)^2 = \frac{5}{4},$$ so the radius of $W_{E_\alpha,E_\beta}$ is smaller than $\sqrt{5}/2$.
\end{proof}

When combined with Proposition \ref{exceptionalBridgelandSeqLemma},  Lemma \ref{nestingLemma} provides the main technical tool we need to complete the proof of Theorem \ref{excBridgeThm}.

\begin{proof}[Proof of Theorem \ref{excBridgeThm}]
We will prove by induction on $q$ that $E_\beta$ is $(s,t)$-semistable along the wall $W_{E_\alpha,E_\beta}$ and that $E_\beta[1]$ is $(s,t)$-semistable along the wall $W_{E_\beta,E_\eta}$.  The conclusion is true for line bundles, so by induction we may assume the result is true for the exceptional slopes $\zeta_i,\alpha,\eta,\omega_i$.  We must show the corresponding walls where these exceptional bundles and their shifts are destabilized are nested inside $W_{E_\alpha,E_\beta}$ or $W_{E_\beta,E_\eta}$ as necessary in each case.

\emph{Case 1: $p \equiv 0 \pmod 4$ or $q= 1$}.  Here we have an exact sequence $$0\to E_{\alpha}\te \Hom(E_\alpha,E_\beta)\to E_\beta \to E_{\zeta_0}[1]\to 0$$ of objects of $\cA_s$ with the same $\mu_{s,t}$-slope for each $(s,t)\in W_{E_\alpha,E_\beta}$.  Decompose $\alpha = \sigma.\tau$ for some exceptional slopes $\sigma$, $\tau$ with $\sigma<\tau$.  By induction, $E_\alpha$ is $(s,t)$-semistable outside the wall $W_{E_\alpha,E_\sigma}$.  We may write $$\sigma = \varepsilon\left(\frac{p'-2\vphantom1}{2^{q'}}\right) \qquad \alpha  = \varepsilon\left(\frac{p'}{2^{q'}}\right) \qquad \alpha.\tau =\varepsilon\left(\frac{p'+1}{2^{q'}}\right) \qquad \tau = \varepsilon\left(\frac{p'+2}{2^{q'}}\right)$$ with $p'\equiv 2 \pmod 4$ and $q'< q$.  By Theorem \ref{kernelSlope} there is an exact sequence $$0\to E_\sigma \to E_\alpha \te \Hom(E_\alpha,E_{\alpha.\tau})\to E_{\alpha.\tau}\to 0, $$ so there is an equality of walls $W_{E_\alpha,E_\sigma} = W_{E_\alpha,E_{\alpha.\tau}}$.  Now the sequence of walls $$W_{E_\alpha,E_{\alpha.\tau}}, W_{E_\alpha,E_{\alpha.(\alpha.\tau)}},W_{E_{\alpha},E_{\alpha.(\alpha.(\alpha.\tau))}},\ldots$$ has each wall nested in the next. But $\beta$ is one of the products $\alpha.\tau,\alpha.(\alpha.\tau),\alpha.(\alpha.(\alpha.\tau)),\ldots$ so we conclude $W_{E_{\alpha},E_\sigma}$ is nested in $W_{E_\alpha,E_\beta}$, and $E_\alpha$ is $(s,t)$-semistable along the wall $W_{E_\alpha,E_\beta}$.

We must also show $E_{\zeta_0}[1]$ is $(s,t)$-semistable along $W_{E_\alpha,E_\beta}$.  From Theorem \ref{kernelSlope} we see $W_{E_{\alpha},E_{\beta}}=W_{E_\alpha,E_{\zeta_0}}$.  Noting $\zeta_0 = \omega_0-3$, the center of $W_{E_{\alpha},E_{\zeta_0}}$ is located at the point $(x,0)$ with $$x = \frac{\alpha+\omega_0-3}{2}+\frac{\Delta_{\omega_0}-\Delta_\alpha}{3+\alpha-\omega_0}=\alpha.\omega_0-\frac{3}{2}=\eta-\frac{3}{2}.$$  We may write $$\sigma' =  \varepsilon\left(\frac{p'-2\vphantom1}{2^{q'}}\right) \qquad \sigma'.\zeta_0 =\varepsilon\left(\frac{p'-1}{2^{q'}}\right) \qquad \zeta_0  = \varepsilon\left(\frac{p'}{2^{q'}}\right) \qquad \tau' = \varepsilon\left(\frac{p'+2}{2^{q'}}\right)$$ with $p'\equiv 2 \pmod 4,$ and observe $\zeta_0=\sigma'.\tau'$.  By induction $E_{\zeta_0}[1]$ is $(s,t)$-semistable outside the wall $W_{E_{\zeta_0},E_{\tau'}}$.  From the exact sequence $$0\to E_{\sigma'.\zeta_0} \to E_{\zeta_0}\te \Hom(E_{\sigma'.\zeta_0},E_{\zeta_0})^*\to E_{\tau'}\to 0$$ we see that $W_{E_{\zeta_0},E_{\tau'}}=W_{E_{\zeta_0},E_{\sigma'.\zeta_0}}$.  By the same argument as in the previous paragraph (making use of the other half of Lemma \ref{nestingLemma}), this wall is nested inside $W_{E_{\zeta_0},E_{\eta-3}}$.  This semicircle lies in the same family of semicircles as $W_{E_{\zeta_0},E_\alpha}$, (and both are to the right of the vertical wall $s = \zeta_0$) so we show that the center of $W_{E_{\zeta_0},E_{\eta-3}}$ lies to the left of the center of $W_{E_{\zeta_0},E_\alpha}$.  This amounts to the inequality $$\frac{\omega_0+\eta}{2}-3+\frac{\Delta_{\omega_0}-\Delta_{\eta}}{\eta-\omega_0}<\eta-\frac{3}{2},$$ which can be rearranged to $$\frac{\omega_0+\eta}{2}+\frac{\Delta_{\omega_0}-\Delta_\eta}{3+\eta-\omega_0}> \eta,$$ i.e. $\eta.\omega_0>\eta$, which is true.

\emph{Case 2: $p\equiv 2\pmod 4$ and $q\geq 2$}.  This case is considerably easier than the previous one.  We have an exact sequence $$0\to E_{\zeta_2}\to E_\alpha\te \Hom(E_{\alpha},E_\beta)\to E_\beta\to 0$$ along the wall $W_{E_{\alpha},E_\beta}$.  By induction, $E_{\alpha}$ is semistable along the wall $W_{E_\alpha,E_{\zeta_2}} = W_{E_\alpha,E_\beta}$.  If we decompose $\zeta_2=\sigma.\tau$ then $E_{\zeta_2}$ is semistable along $W_{E_\sigma,E_{\zeta_2}}$.  This wall equals $W_{E_{\zeta_2},E_{\zeta_2.\tau}}$, so $W_{E_\sigma,E_{\zeta_2}}$ is nested inside $W_{E_{\zeta_2},E_\alpha} = W_{E_{\alpha}.E_\beta}$.  Thus both $E_\alpha$ and $E_{\zeta_2}$ are semistable along $W_{E_\alpha,E_\beta}$.

Shifted objects can be handled in the same manner.      
\end{proof}

\begin{proof}[End of the proof of Theorem \ref{bridgelandThm}]
By Theorem \ref{excBridgeThm} and Corollary \ref{root5Cor}, any exceptional bundle is $(s,t)$-semistable outside a wall of radius smaller than $\sqrt{5}/2$. It remains to show the same is true for $W[1]$.  Using the notation from this section, we can choose integers $k_1, k_2, p,q$ such that $W$ has a resolution of the form 
$$0 \to W \to E_{\zeta_0}^{k_1}\to E_{\alpha}^{k_2}\to 0,$$ and we may assume $p\equiv 0 \pmod 4$.  We obtain an exact sequence $$0\to E_{\alpha}^{k_2}\to W[1]\to E_{\zeta_0}^{k_1}[1]\to 0$$ valid in $\cA_s$ for any $s$ with $\zeta_0 < s< \alpha$ (noting that $\mu(W)\leq \zeta_0$).    Case 1 of the proof of Theorem \ref{excBridgeThm} shows $E_\alpha$ and $E_{\zeta_0}[1]$ are semistable along the wall $W_{E_\alpha,E_{\zeta_0}}=W_{E_\alpha,E_\beta}$.  But this wall has radius smaller than $\sqrt{5}/2$ by Corollary \ref{root5Cor}.
 \end{proof}

\bibliographystyle{plain}

\begin{thebibliography}{ABCH}

\bibitem[AP]{AbramPol}
D. Abramovich\ and\ A. Polishchuk, Sheaves of $t$-structures and valuative criteria for stable complexes, J. Reine Angew. Math. {\bf 590} (2006), 89--130.

\bibitem[ACK]{ConsulKing}
L. \'Alvarez-C\'onsul\ and\ A. King, A functorial construction of moduli of sheaves, Invent. Math. {\bf 168} (2007), no.~3, 613--666.

\bibitem[AB]{ArcaraBertram}
D. Arcara, A. Bertram. Bridgeland-stable moduli spaces for $K$-trivial surfaces, with an appendix by Max Lieblich, to appear J. Eur. Math. Soc.

\bibitem[ABCH]{ABCH}
D. Arcara, A. Bertram, I. Coskun, and J. Huizenga.
\newblock The minimal model program for the Hilbert scheme of points on $\mathbb{P}^2$ and Bridgeland stability.  Preprint.


\bibitem[BM]{BayerMacri}
A. Bayer and\ E. Macr\`i, The space of stability conditions on the local projective plane, Duke Math. J. {\bf 160} (2011), no.~2, 263--322.

\bibitem[BM2]{BayerMacri2}
A.~Bayer and E.~Macr\`{i}. Projectivity and birational geometry of Bridgeland
moduli spaces.  Preprint.

\bibitem[Brm1]{Brambilla2}
M. C. Brambilla, Cokernel bundles and Fibonacci bundles, Math. Nachr. {\bf 281} (2008), no.~4, 499--516. 

\bibitem[Brm2]{Brambilla1}
M. C. Brambilla, Simplicity of generic Steiner bundles, Boll. Unione Mat. Ital. Sez. B Artic. Ric. Mat. (8) {\bf 8} (2005), no.~3, 723--735. 

\bibitem[Bri]{Bridgeland}
T. Bridgeland, Stability conditions on $K3$ surfaces, Duke Math. J. {\bf 141} (2008), no.~2, 241--291.

\bibitem[C]{MO3} H. Cohn, Last term of repeating continued fraction expansion, URL (2012) mathoverflow.net/q/106217.

\bibitem[Da]{Davenport}
H. Davenport, {\it The higher arithmetic}, eighth edition, Cambridge Univ. Press, Cambridge, 2008. 

\bibitem[DW]{DerksenWeyman}
Derksen, Harm; Weyman, Jerzy. Semi-invariants of quivers and saturation for Littlewood-Richardson coefficients. J. Amer. Math. Soc. 13 (2000), no. 3, 467--479 (electronic).

\bibitem[Dr1]{DrezetBeilinson}
J.-M. Drezet, Fibr\'es exceptionnels et suite spectrale de Beilinson g\'en\'eralis\'ee sur ${\bf P}\sb 2({\bf C})$, Math. Ann. {\bf 275} (1986), no.~1, 25--48.

\bibitem[Dr2]{Drezet}
J.-M. Drezet, Fibr\'es exceptionnels et vari\'et\'es de modules de faisceaux semi-stables sur ${\bf P}\sb 2({\bf C})$, J. Reine Angew. Math. {\bf 380} (1987), 14--58. 



\bibitem[DLP]{DLP}
J.-M. Drezet\ and\ J. Le Potier, Fibr\'es stables et fibr\'es exceptionnels sur ${\bf P}\sb 2$, Ann. Sci. \'Ecole Norm. Sup. (4) {\bf 18} (1985), no.~2, 193--243. 
\bibitem[E]{Eisenbud}
D. Eisenbud, {\it The geometry of syzygies}, Graduate Texts in Mathematics, 229, Springer, New York, 2005. 



\bibitem[F1]{Fogarty1}
J. Fogarty, Algebraic families on an algebraic surface, Amer. J. Math {\bf 90} (1968), 511--521. 

\bibitem[F2]{Fogarty2}
J. Fogarty, Algebraic families on an algebraic surface. II. The Picard scheme of the punctual Hilbert scheme, Amer. J. Math. {\bf 95} (1973), 660--687. 

\bibitem[GH]{GottscheHirschowitz} L. G\"ottsche\ and\ A. Hirschowitz, Weak Brill-Noether for vector bundles on the projective plane, in {\it Algebraic geometry (Catania, 1993/Barcelona, 1994)}, 63--74, Lecture Notes in Pure and Appl. Math., 200 Dekker, New York.





\bibitem[LP]{LePotierLectures}
J. Le Potier, {\it Lectures on vector bundles}, translated by A. Maciocia, Cambridge Studies in Advanced Mathematics, 54, Cambridge Univ. Press, Cambridge, 1997. 

\bibitem[L]{lieblich}
M. Lieblich, Moduli of complexes on a proper morphism, J. Algebraic Geom. {\bf 15} (2006), no.~1, 175--206.


\bibitem[H1]{MO1} J. Huizenga, Homeomorphisms of the rationals, URL (2012) mathoverflow.net/q/105758.

\bibitem[H2]{HuizengaPaper}
J. Huizenga, Restrictions of Steiner bundles and divisors on the Hilbert scheme of points in the plane, Int. Math. Res. Not. (2012).

\bibitem[H3]{thesis}
J. Huizenga, {\em Restrictions of Steiner bundles and divisors on the Hilbert scheme of points in the plane.}  Ph.D. thesis, Harvard University, 2012.

\bibitem[O]{Ottaviani}
G. Ottaviani, {\em Variet\`a proiettive di codimensione piccola}, Ist. nazion. di alta matematica F. Severi, 2, Aracne, Rome, 1995.


\bibitem[Sa]{MO2} D. Savitt, Palindromic continued fraction, URL (2012) mathoverflow.net/q/106279.

\bibitem[Sc]{Schofield}
A. Schofield, Semi-invariants of quivers, J. London Math. Soc. (2) {\bf 43} (1991), no.~3, 385--395.

\bibitem[SvdB]{SvdB}
A. Schofield\ and\ M. van den Bergh, Semi-invariants of quivers for arbitrary dimension vectors, Indag. Math. (N.S.) {\bf 12} (2001), no.~1, 125--138. 

\bibitem[T]{toda}
Y. Toda, Moduli stacks and invariants of semistable objects on $K3$ surfaces, Adv. Math. {\bf 217} (2008), no.~6, 2736--2781.

\end{thebibliography}

\end{document}